\documentclass[10pt]{amsart}
\usepackage{latexsym}
\usepackage{amsfonts}
\usepackage{amssymb}
\usepackage{amsmath}
\usepackage{amscd}
\usepackage{graphicx}
\usepackage{eucal}
\usepackage{amsthm}
\usepackage[all]{xy}
\usepackage{pdfsync}
\usepackage{xspace}

\input xy
\xyoption{all}
\xyoption{2cell}

\UseAllTwocells
\CompileMatrices

\usepackage{}
\usepackage{PageSetup}
\usepackage{Alphabets}
\usepackage{Definitions}
\usepackage{Environments}
\usepackage{TikzSetup}
\usepackage{AuthorInfo}

\theoremstyle{plain} % bold environment name, italic text
\clevertheorem{theorem}{Theorem}{Theorems}
\newtheorem*{theorem*}{Theorem}
\clevertheorem{proposition}{Proposition}{Propositions}
\clevertheorem{lemma}{Lemma}{Lemmatta}
\clevertheorem{corollary}{Corollary}{Corollaries}
\clevertheorem{conjecture}{Conjecture}{Conjectures}

\theoremstyle{definition} % bold environment name, plain text
\clevertheorem{definition}{Definition}{Definitions}
\newtheorem*{definition*}{Definition}
\clevertheorem{example}{Example}{Examples}
\newtheorem*{example*}{Example}

\clevertheorem{construction}{Construction}{Constructions}
\newtheorem*{construction*}{Construction}

\clevertheorem{visualization}{Visualization}{Visualizations}
\newtheorem*{visualization*}{Visualization}

\newtheorem*{question*}{Question}

\newtheorem*{claim*}{Claim}

\begin{document}
   \title{The Inductive Coherator for Grothendieck Infinity Groupoids}

\author{Johnathon Taylor}
\address{
Department of Mathematics, Applied Mathematics, and Statistics,
Case Western Reserve University,
10900 Euclid Ave, Cleveland, OH 44106-2624 
}
\email{jmt240@case.edu}
\begin{abstract}
We extend the theory of distributive series of monads of \cite{EC1} by extending the definition to include an $\bN$-indexed collection of monads. Under certain conditions, distributive series of monads will have a colimit in the category of pointed endofunctors. We define a \emph{completable} distributive series of monads to be a distributive series of monads whose induced pointed endofunctor, if it exists, lifts to a monad. We then construct factorization systems used to generate monads on the category of theories over $\Theta_0^\op$, in order to form two \emph{completable} distributive series of monads. The first completable distributive series of monads induces a monad that sends the identity theory over $\Theta_0^\op$ to an $(\infty,0)$-coherator whose inductive construction mimics inductive weak enrichment. The second completable distributive series of monads induces a monad that sends the identity theory over $\Theta_0^\op$ to a theory for strict $\infty$-groupoids.
\end{abstract}

\maketitle

\tableofcontents
	
\section{Introduction}
One way that higher category theory is approached in the globular setting is through the lens of weak enrichment. A \emph{weak $(n+1)$-category} is, by definition, a category weakly enriched in weak $n$-categories, where $n>0$. In principle, one can write down a presentation for an $n$-category for every $n>0$, but in practice such presentations quickly become unwieldy, and essentially intractable once $n \geq 4$ %%(see \cite{Trimb}). We make matters more complicated when we require that all $n$-cells be invertible for $n>0$. Nevertheless, the perspective of iterated weak enrichment remains valuable: it provides a way to inductively assemble the structure data, building it up one dimension at a time.

Another way that higher category theory is tackled in the globular setting is by building a theory over a limit sketch and defining %%\emph{Grothendieck $\infty$-groupoids} to be models over that theory (see Definition 2.5 of %%\cite{Dim1}). The theory approach was first suggested by Grothendieck in \emph{Pursuing Stacks} \cite{Gr1} and rigorously introduced by Maltsiniotis in \cite{Malt}. The theory method in its current form in the literature has no attached universal property, is not obtained algebraicly, and is not defined inductively. However, the theory method has two upsides:
\begin{itemize}
\item there is a straightforward built-in way to systematically build a map that compares models over the theory to other categories
\item a solution to the Homotopy Hypothesis has been boiled down to showing that the pushout of a finitely cellular $\infty$-groupoid along a generating trivial cofibration is a weak equivalence (see Section 5.3 of \cite{hen2}).
\end{itemize}

A third method in the globular setting is the method of globular operads of Batanin and Leinster (see Definition 8.7 of \cite{Bat1} and Definition 9.3.6 of \cite{Lein}). Leinster defines a weak $n$-category to be an algebra over the initial $n$-globular operad with contraction (see Proposition 9.3.5 and Definition 9.3.6 of \cite{Lein}). Leinster's definition is algebraic in nature, has a universal property, and is inductively defined. 

In this paper, we use Grothendieck $\infty$-groupoids. We build a theory that removes the technical obstruction to working with general Grothendieck $\infty$-groupoids: the data for the theory not necessarily being encoded in a structured fashion.  In this paper, we use an expansion of the theory for distributive series of monads and the theory of factorization systems to obtain a method that inductively builds a special choice of theory, $IC$, for Grothendieck $\infty$-groupoids that we call the \emph{inductive coherator}. We understand this inductive construction as applying weak enrichment inductively on the dimensions of cells, is constructed algebraically, and is obtained by a universal property. Moreover, we use the same methodology to obtain a theory for strict $\infty$-groupoids. 

We organize the rest of the paper as follows. In Section 2, we provide background on the theory of monads, comonads, distributive laws, factorization systems, globular sets, and Grothendieck $\infty$-groupoids. In Section 3, we obtain \emph{Kelly's Small Object Argument} (see Theorem \ref{Kelly_SOA}), which is reminiscent of Kelly's $k_+$-construction (see Theorem 11.5 and 11.3 of \cite{Kell}), as a quotient of the \emph{Algebraic Small Object Argument} of Garner.  In Section 4, we provide background on the theory of distributive series of monads as introduced by Cheng in \cite{EC1} and expand upon the definition to have an $\bN$-indexed collection of monads.
 
 We establish our choice of theory $IC$ (see Theorem \ref{our_sp_choice}) for Grothendieck $\infty$-groupoids in Section 5 by building a completable distributive series of monads (see Theorem \ref{ind_monad_AWFS}). We conclude our main exposition in Section 6 by constructing a theory for strict $\infty$-groupoids (see Theorem \ref{strict_coh_is_true}). We provide two appendices that serve as Section 7 and Section 8, respectively. The first appendix, Appendix A, sketches out the structure maps obtained in the inductive coherator $IC$ at each step.  The second appendix, Appendix B, covers the theory of distributive series of monads when specialized to monads whose units are epimorphism on components.

\section{Background}
We begin by providing and setting notation that will be used throughout the rest of this paper about monads and distributive laws between monads. Then, we discuss comonads and distributive laws between monads and comands. Next, we provide background on lifts and factorization systems. Finally, we restate definitions and theorems for the theory of Grothendieck $\infty$-groupoids. We take the definitions and lemmas provided from various sources (see \cite{Beck_Dist}, \cite{Bo2}, \cite{Hov2}, and \cite{Quill}).

\begin{notation}
Given a category $C$, we denote
\begin{itemize}
\item a general monad on $C$ by $(R,\eta,\mu)$
\item the category of monads on $C$ by $\Monad(C)$
\item the category of endofunctors on $C$ by $\End(C)$
\item the category of pointed endofunctors on $C$ by $\End_*(C)$, and
\item the forgetful functor from monads to pointed endo-functor as 
\[
U:\Monad(C)\to \End_*(C).
\]
\end{itemize}
\end{notation}

Later in this paper, we want to talk about when a pointed-endofunctor on a category lifts to a monad structure.

\begin{definition}
Let $(T,\eta)$ be a pointed endofunctor on $C$. We say that $(T,\eta)$ \emph{lifts} to a monad if there is a monad $(T,\eta,\mu)$ such that
\[
U(T,\eta,\mu)=(T,\eta).
\]
It \emph{uniquely lifts} if given a monad $(T,\eta,\mu')$ with $U(T,\eta,\mu')=(T,\eta)$, then $\mu=\mu'$.
\end{definition}

\begin{definition} [Beck, Definition 1 of \cite{Beck_Dist}]\label{mon_over_mon}
	Let $(S,\eta^S,\mu^S)$ and $(T,\eta^T,\mu^T)$ be monads on a category $C$.  A {\em distributive law of $S$ over $T$} consists of a natural transformation $\lambda : ST \Rightarrow TS$ such that the following diagrams commute.
	
	\[
		\begin{tikzpicture}[node distance=1.5cm]
			\node (A){$ST$};
			\node (B)[right of=A]{};
			\node(C)[right of=B]{$TS$};
			\node(D)[below of=B] {$T$};
			\node (E)[above of=B]{$S$};
			\draw[->] (A) to node[black,above=3] {$\lambda$} (C);
			\draw[->,swap](D) to node [black,left=3,below=3] {$\eta^ST$} (A);
			\draw[->] (D) to node [black,right=3] {$T\eta^S$} (C);
			\draw[->,swap](E) to node [black,left=3] {$S\eta^T$} (A);
			\draw[->] (E) to node [black,right=3] {$\eta^TS
				$} (C);
		\end{tikzpicture}
	\]
	
	\[
		\begin{tikzpicture}[node distance=1.5cm]
			\node (A) {$S^2T$};
			\node (B)[below of =A]{$ST$};
			\node (C)[below of=B]{$ST^2$};
			\node (D)[right of=C,node distance=3cm]{$STS$};
			\node (E) [above of=D]{};
			\node (F) [above of=E]{$TST$};
			\node (G) [right of=F,node distance=3cm]{$TS^2$};
			\node (H) [below of=G]{$TS$};
			\node (I)[below of=H]{$T^2S$};
			\draw[->,swap] (A) to node[black,left=3] {$\mu^ST$}(B);
			\draw[->] (A) to node [black,above=3] {$S\lambda$}(F);
			\draw[->] (F) to node[black,above=3] {$\lambda S$}(G);
			\draw[->](G) to node[black,right=3] {$T\mu_S$}(H);
			\draw[->](B) to node[black,above=3] {$\lambda$}(H);
			\draw[->] (C) to node[black,left=3] {$S\mu^T$}(B);
			\draw[->] (I) to node[black, right=3]{$\mu^TS$} (H);
			\draw[->,swap](C) to node[black, below=3]{$\lambda T$} (D);
			\draw[->,swap] (D) to node[black,below=3]{$T\lambda$}(I);
		\end{tikzpicture}
\]
\end{definition}

\begin{proposition}(Beck, Proposition 1 of \cite{Beck_Dist})\label{comp_monad}
Let $(S,\eta^S,\mu^S)$, $(T,\eta^T,\mu^T)$ be monads on $C$ and $\lambda:ST\Rightarrow TS$ be a distributive law. Then the triple 
\[
(ST,S\eta^T\circ \eta^S,S\mu^T\circ\mu^ST^2\circ S\lambda T)
\]
is a monad.
\end{proposition}

That is to say, the composite of the underlying pointed endo-functors lifts to a monad structure. 

\begin{notation}
Given a category $C$, we denote
\begin{itemize}
\item a general comonad on $C$ by $(L,\epsilon,\Delta)$, and
\item the category of comonads on $C$ by $\Comonad(C)$.
\end{itemize}
\end{notation}

\begin{definition} (Power and Watanabe, Definition 6.1 of \cite{PowWat})\label{common_over_mon}
	Let $(S,\epsilon^S,\Delta^S)$ be a comonad and $(T,\eta^T,\mu^T)$ be a monad on a category $\cC$.  A {\em distributive law of $S$ over $T$} consists of a natural transformation $\lambda : ST \Rightarrow TS$ such that the following diagrams commute:
	
\[
		\begin{tikzpicture}[node distance=1.5cm]
			\node (A){$ST$};
			\node (B)[right of=A]{};
			\node(C)[right of=B]{$TS$};
			\node(D)[below of=B] {$T$};
			\node (E)[above of=B]{$S$};
			\draw[->] (A) to node[black,above=3] {$\lambda$} (C);
			\draw[->,swap](A) to node [black,left=3,below=3] {$\epsilon^ST$} (D);
			\draw[->] (C) to node [black,right=3] {$T\epsilon^S$} (D);
			\draw[->,swap](E) to node [black,left=3] {$S\eta^T$} (A);
			\draw[->] (E) to node [black,right=3] {$\eta^TS
				$} (C);
		\end{tikzpicture}
\]
\[
		\begin{tikzpicture}[node distance=1.5cm]
			\node (A) {$S^2T$};
			\node (B)[below of =A]{$ST$};
			\node (C)[below of=B]{$ST^2$};
			\node (D)[right of=C,node distance=3cm]{$STS$};
			\node (E) [above of=D]{};
			\node (F) [above of=E]{$TST$};
			\node (G) [right of=F,node distance=3cm]{$TS^2$};
			\node (H) [below of=G]{$TS$};
			\node (I)[below of=H]{$T^2S$};
			\draw[->,swap] (B) to node[black,left=3] {$\Delta^ST$}(A);
			\draw[->] (A) to node [black,above=3] {$S\lambda$}(F);
			\draw[->] (F) to node[black,above=3] {$\lambda S$}(G);
			\draw[->](H) to node[black,right=3] {$T\Delta^S$}(G);
			\draw[->](B) to node[black,above=3] {$\lambda$}(H);
			\draw[->] (C) to node[black,left=3] {$S\mu^T$}(B);
			\draw[->] (I) to node[black, right=3]{$\mu^TS$} (H);
			\draw[->,swap](C) to node[black, below=3]{$\lambda T$} (D);
			\draw[->,swap] (D) to node[black,below=3]{$T\lambda$}(I);
		\end{tikzpicture}
\]
\end{definition}

We require distributive laws of comonads over monads for the definition of algebraic weak factorization systems (see Definition \ref{AWFS}).

\begin{notation}
Let $C$ be a category. We write $\Arr(C)$ to denote the category of squares in $C$.
\end{notation}

\begin{definition}
	Consider a commutative square 
	\[
	\begin{tikzpicture}[node distance=2cm,auto]
		\node (A) {$x$};
		\node (B)[right of=A] {$a$};
		\node(C)[below of=A] {$y$};
		\node (D)[right of=C] {$b$};
		\draw[->] (A) to node {$i$}(B);
		\draw[->,swap] (A) to node{$f$}(C);
		\draw[->] (B) to node{$g$}(D);
		\draw[->,swap] (C) to node {$j$}(D);
	\end{tikzpicture}
	\]
	in $C$. A \emph{lifting} for the square is a morphism $h:y\to a$ such that the two triangles in the diagram
	\[
	\begin{tikzpicture}[node distance=2cm,auto]
		\node (A) {$x$};
		\node (B)[right of=A] {$a$};
		\node(C)[below of=A] {$y$};
		\node (D)[right of=C] {$b$};
		\draw[->] (A) to node {$i$}(B);
		\draw[->,swap] (A) to node{$f$}(C);
		\draw[->] (B) to node{$g$}(D);
		\draw[->,swap] (C) to node {$j$}(D);
		\draw[->,dotted](C) to node[left=3]{$h$}(B);
	\end{tikzpicture}
	\]
	commute.
\end{definition}

We call a problem of finding a lift a \emph{lifting problem} and the lift is called a \emph{solution}.

\begin{definition}
	Let $f:x\to y$ and $g:a\to b$ be two morphisms in $C$. We say that the ordered pair $(f, g)$ is \emph{orthogonal}, denoted by $f\perp g$, provided that any commutative diagram
	\[
	\begin{tikzpicture}[node distance=2cm,auto]
		\node (A) {$x$};
		\node (B)[right of=A] {$a$};
		\node(C)[below of=A] {$y$};
		\node (D)[right of=C] {$b$};
		\draw[->] (A) to node {}(B);
		\draw[->,swap] (A) to node{$f$}(C);
		\draw[->] (B) to node{$g$}(D);
		\draw[->,swap] (C) to node {}(D);
	\end{tikzpicture}
	\]
	admits a lifting.
\end{definition}

\begin{notation}
Let $C$ be a category and $H$ be a set of maps of $C$. We write 
\[
H^{\perp}:=\{f\in\mor(C):h\perp f\text{ for all }h\in H\}
\]
and 
\[
^{\perp} H:=\{f\in\mor(C):f\perp h\text{ for all }h\in H\}
\]
\end{notation}

\begin{definition}
	A \emph{weak factorization system} (or \emph{wfs} for short) on a category $C$ consists of a pair $(\zL,\zR)$ of classes of morphisms of $C$ such that:
	\begin{itemize}
		\item every morphism $f:x\to y$ may be factored as $f=p\circ i$ where $p\in\zR$ and $i\in\zL$
		\item $^\perp\zR\subset \zL$ and $\zL^\perp\subset \zR$.
	\end{itemize}
\end{definition}

\begin{definition}
	Let $\lambda$ be an ordinal and $x:\lambda\to C$ be a functor. 
	The map
$$x_0\longrightarrow {\rm colim}_{i\geq 0} \; x_i$$	
induced by the directed colimit is called the \emph{transfinite composition} of the sequence.
\end{definition}

\begin{definition}
	Let $\mathcal{I}$ be a class of morphisms in $C$. We denote by $\mathcal{I}\mbox{-cell}$ the class of morphisms formed by transfinite compositions of pushouts of coproducts of elements in $\mathcal{I}$. In other words, a morphism in $\mathcal{I}\mbox{-cell}$ is the transfinite composition of a sequence,
	$$x_0\stackrel{\phi_0}\longrightarrow x_1 \stackrel{\phi_1}\longrightarrow x_2\longrightarrow \cdots$$
	where each $\phi_i$ fits into a pushout diagram
	\[
	\begin{tikzpicture}[node distance=2cm,auto]
		\node (A) {$ \coprod_{y\in \Lambda_i} a_y$};
		\node (B)[right of=A] {$x_i$};
		\node(C)[below of=A] {$\coprod_{y\in \Lambda_i} b_y$};
		\node (D)[right of=C] {$x_{i+1}$};
		\draw[->] (A) to node {}(B);
		\draw[->,swap] (A) to node{$\coprod_{x\in \Lambda_i} f_x$}(C);
		\draw[->] (B) to node{$\phi_i$}(D);
		\draw[->,swap] (C) to node {}(D);
	\end{tikzpicture}
	\]
	with each $f_y:a_y\to b_y\in \mathcal{I}$. 
\end{definition}

\begin{definition}
	An object $A$ is \emph{sequentially small}, if for any sequence 
	\[
	X_0\rightarrow X_1 \rightarrow X_2\rightarrow \cdots 
	\]
	of morphisms in $C$, the induced map
	$${\rm colim}_{i\geq 0}\; C(A, X_i)\longrightarrow C(A, {\rm colim}_{i\geq 0} \; X_i)$$
	is a bijection.
\end{definition}

\begin{theorem}(The Small Object Argument)
	Suppose $C$ is a cocomplete category and $\mathcal{I}$ is a set of  morphisms whose domains are sequentially small. Then any morphism $f$ admits a factorization $f=p\circ i$ with $i\in \mathcal{I}\mbox{-}{\rm cell}$ and $p\in \mathcal{I}^\perp$. 
\end{theorem}

We have cited this theorem from Hovey (see Theorem 2.1.14 of \cite{Hov2}) but it originates from Quillen (see Lemma 2.3.3 of \cite{Quill}).

\begin{definition}
	If a weak factorization system on a category $C$ is generated by the small object argument, then we say that the weak factorization system on $C$ is \emph{cofibrantly generated}.
\end{definition}

\begin{definition}\label{fun_fact}
	A \emph{functorial factorization system} on $C$ is a functor 
	\[
	F:\Arr(C)\to \Arr(C)\times_C\Arr(C)
	\]
that is a section to the composition functor $c:\Arr(C)\times_C\Arr(C)\to \Arr(C)$ . We write $F=(L,E,R)$ where $L$, $E$, and $R$ are obtained by applying the projections. 
\end{definition}

\begin{notation}
	The \emph{globe category} $\bG$ is the category
\[
		\begin{tikzpicture}[node distance=2cm]
		\node (A) {$0$};
			\node (B)[right of=A]{$1$};
			\node (C)[right of=B]{$2$};
			\node (D)[right of=C]{$\cdots$};
	\draw[transform canvas={yshift=0.3ex},->] (A) to node[above=3] {$s$} (B);
			\draw[transform canvas={yshift=-0.3ex},->,swap](A) to node[below=3] {$t$} (B);
			\draw[transform canvas={yshift=0.3ex},->] (B) to node[above=3] {$s$} (C);
			\draw[transform canvas={yshift=-0.3ex},->,swap](B) to node[below=3] {$t$} (C);
			\draw[transform canvas={yshift=0.3ex},->] (C) to node[above=3] {$s$} (D);
			\draw[transform canvas={yshift=-0.3ex},->,swap](C) to node[below=3] {$t$} (D);
	\end{tikzpicture}
	\]	
	subject to the relations 
	\[
	s\circ s=t\circ s
	\]
	and 
	\[
	t\circ t=s\circ t.
	\]
\end{notation}

\begin{definition}
	A \emph{globular object} $X$ in a category $C$ is a $C$-enriched presheaf over $\bG$. This means a functor $X:\bG^\op\to C$.
\end{definition}

\begin{example}
Define a functor $D :\bG \to \Top$ by letting $D$ send $i$ to the $i$-dimensional ball
\[ D^i = \{x \in \mathbb{R}^i : ||x|| \le 1\} \] for $i\geq 0$
For $i \ge 1$, the morphisms $\sigma_i$ and $\tau_i$ are sent by $D$
to $\sigma^i$ and $\tau^i$, respectively, defined by
\[ \sigma^i(x) = (x, \sqrt{1 - \Vert x\Vert^2}) \quad\text{and}\quad
\tau^i(x) = (x, -\sqrt{1 - \Vert x\Vert^2}) 
\]
for $x \in D^{i-1}$. These morphisms are the inclusions of the two hemispheres of $D^{i}$ into $D^i$. It is easy to verify that $D$ thus well-defined. This induces a globular object 
\[
\Top(D^{(-)},X):\bG^\op\to \Top
\]
 of $\Top$ for every space $X$.
\end{example}

\begin{definition}
Consider diagrams of the form

\begin{center}
	\begin{tikzpicture}[node distance=1cm, auto]
		
		\node (A) {$n_1$};
		\node (B) [right of=A]{} ;
		\node (C) [above of=B] {$n_2$};
		\node (D) [below of=C]{};
		\node(E)[right of=D]{$n_3$};
		\node (F) [right of=E]{} ;
		\node (G) [above of=F] {$n_4$};
		\node (H) [right of=G]{} ;
		\node (I) [below of=H] {$n_5$};
		\node (J)[right of=I]{};
		\node (X) [node distance=0.5cm, above of=J] {$\cdots$};
		
		\node (K)[right of=J] {$n_{2k-1}$};
		\node (L) [right of=K]{} ;
		\node (M) [above of=L] {$n_{2k}$};
		\node (N) [below of=M]{};
		\node(O)[right of=N]{$n_{2k+1}$};
		
		\draw[->,swap] (C) to node {$s$} (A);
		\draw[->] (C) to node {$t$} (E);
		\draw[->,swap] (G) to node {$s$} (E);
		\draw[->] (G) to node {$t$} (I);
		
		\draw[->,swap] (M) to node {$s$} (K);
		\draw[->] (M) to node {$t$} (O);
		
	\end{tikzpicture}
\end{center}
in $\bG$. Part of the diagram data is a $(2k+1)$-tuple $\overline{n}=(n_1,\dots,n_{2k+1})$ such that 
\[
n_{2i-1}>n_{2i}<n_{2i+1}
\]
for all $1\leq i\leq k$. This interpretation of the diagram is called a \emph{table of dimensions}.
\end{definition}

\begin{definition}
	If $A:\bG^\op\to C$ is a functor, then there is an induced diagram 
	
	\begin{center}
		\begin{tikzpicture}[node distance=1.25cm, auto]
			
			\node (A) {$An_1$};
			\node (B) [right of=A]{} ;
			\node (C) [above of=B] {$An_2$};
			\node (D) [below of=C]{};
			\node(E)[right of=D]{$An_3$};
			\node (F) [right of=E]{} ;
			\node (G) [above of=F] {$An_4$};
			\node (H) [right of=G]{} ;
			\node (I) [below of=H] {$An_5$};
			\node (J)[right of=I]{};
			\node (X) [node distance=0.75cm, above of=J] {$\cdots$};
			
			\node (K)[right of=J] {$An_{k-2}$};
			\node (L) [right of=K]{} ;
			\node (M) [above of=L] {$An_{k-1}$};
			\node (N) [below of=M]{};
			\node(O)[right of=N]{$An_k$};
			
			\draw[->] (A) to node {$As$} (C);
			\draw[->,swap] (E) to node {$At$} (C);
			\draw[->] (E) to node {$As$} (G);
			\draw[->,swap] (I) to node {$At$} (G);
			
			\draw[->] (K) to node {$As$} (M);
			\draw[->,swap] (O) to node {$At$} (M);
			
		\end{tikzpicture}
	\end{center}
	and if it has a limit, the limit is called the \emph{globular product}. 
\end{definition}

\begin{definition}
	Let $\Theta_0$ be the category whose objects are all the tables of dimensions and 
	\[
	\Theta_0(\vec{n},\vec{m})=[\bG^\op,\Set](Y(\vec{n}),Y(\vec{m})),
	\]
	where $Y$ is the Yoneda embedding. Moreover, given a table of dimensions $\vec{n}=(n_1,\dots,n_k)$, we define its \emph{height} to be 
	\[
	\hgt(\vec{n})=\max\{n_1,\dots,n_k\}.
	\]
\end{definition}
  
\begin{lemma}(Bourke, Lemma 2.1 of \cite{Bo2})
	There is a functor $D:\bG^\op\to\Theta_0^\op$ where $D(n)=(n)$ on objects which satisfies the following universal property: if $C$ is a category admitting $A$-globular products, there exists an essentially unique extension
	\[
	\begin{tikzpicture}[node distance=2cm, auto]
		
		\node (A) {$\bG^\op$};
		\node (B) [right of=A]{$C$} ;
		\node (C) [above of=A]{$\Theta_0^\op$};

		\draw[->,swap] (A) to node {$A$} (B);
		\draw[->](A) to node {$D^\op$}(C);
		\draw[->](C) to node {$A'$}(B);
		
	\end{tikzpicture}
	\]
	of $A$ to a globular product preserving functor $A'$. This sends $\overline{n}$ to the associated globular product.
\end{lemma}

\begin{definition}
	An \emph{extension over $\Theta_0^\op$} is a functor $H:\Theta_0^\op\to C$ which preserves globular products. A map of extensions from $H:\Theta_0^\op\to C$ to another $K:\Theta_0^\op\to D$ is a map $\gamma:C\to D$ that satisfies
	\[
	\gamma\circ H=K.
	\]
	\end{definition}
	
\begin{notation}
	We write $\Ex_{\Theta_0^\op}$ to denote the category of extensions over $\Theta_0^\op$ and maps between them. Moreover, we write $\Th_{\Theta_0^\op}$ to denote the full subcategory of $\Ex_{\Theta_0^\op}$ whose objects are identity on objects functors. The objects of $\Th_{\Theta_0^\op}$ will be called \emph{globular theories}.
\end{notation}

\begin{remark}
	The categories $\Ex_{\Theta_0^\op}$ and $\Th_{\Theta_0^\op}$ are complete, cocomplete, and locally presentable. 
\end{remark}

\begin{definition}
	Let $A$ be a globular object of $C$.  An \emph{admissible pair} of $n$-cells in $A$ is a pair 
	\[
		\begin{tikzpicture}[node distance=2cm]
		\node (A) {$X$};
			\node (B)[right of=A]{$A(n)$};
	\draw[transform canvas={yshift=0.3ex},->] (A) to node[above=3] {$f$} (B);
			\draw[transform canvas={yshift=-0.3ex},->,swap](A) to node[below=3] {$g$} (B);
	\end{tikzpicture}
	\]
	of maps in $C$ where $X$ is a globular sum of $A$ and $\hgt(X)\leq n+1$ such that either $n=0$ or $s\circ f=s\circ g$ and $t\circ f=t\circ g$.
\end{definition}

\begin{definition}	
	Let $A$ be a globular object of $C$. A \emph{lift} for an admissible pair is an arrow $\delta_{f{,}g}:X\to A(n+1)$ such that
	\[
		\begin{tikzpicture}[node distance=2cm]
			\node (A) {$X$};
			\node (B)[right of=A]{$A(n)$};
			\node (C) [above of=B]{$A(n+1)$};
			\draw[transform canvas={yshift=0.3ex},->] (A) to node[above=3] {$f$} (B);
			\draw[transform canvas={yshift=-0.3ex},->,swap](A) to node[below=3] {$g$} (B);
			\draw[transform canvas={xshift=-0.3ex},->] (C) to node[left=3] {$s$} (B);
			\draw[transform canvas={xshift=0.3ex},->,swap](C) to node[right=3] {$t$} (B);
			\draw[->](A) to node[above=4] {$\delta_{f,g}$}(C);
		\end{tikzpicture}
		\]
	commutes. 
\end{definition}

\begin{definition}
The \emph{dimension} of an admissable pair of the form above is $\dim(f,g)=n$.
\end{definition}

\begin{definition}
	We say $A$ is \emph{contractible} if every admissable pair  has a lift.
\end{definition}

\begin{definition}
	We say that a globular theory  $J:\Theta_0^\op\to C$ is \emph{contractible} if it is contractible as a globular object of $C$.
\end{definition}

\begin{lemma}
	Given a map $F:C\to D$ of globular theories, $F$ maps admissible pairs to admissible pairs.
\end{lemma}

\begin{proof}
Suppose 
\[
		\begin{tikzpicture}[node distance=2cm]
		\node (A) {$\vec{p}$};
			\node (B)[right of=A]{$k$};
	\draw[transform canvas={yshift=0.3ex},->] (A) to node[above=3] {$f$} (B);
			\draw[transform canvas={yshift=-0.3ex},->,swap](A) to node[below=3] {$g$} (B);
	\end{tikzpicture}
	\]
is an admissible pair in $C$. Then 
\[
s\circ F(f)=F(s\circ f)=F(s\circ g)=s\circ F(g)
\]
and 
\[
t\circ F(f)=F(t\circ f)=F(t\circ g)=t\circ F(g),
\]
so that 
\[
		\begin{tikzpicture}[node distance=2cm]
		\node (A) {$\vec{p}$};
			\node (B)[right of=A]{$k$};
	\draw[transform canvas={yshift=0.3ex},->] (A) to node[above=3] {$Ff$} (B);
			\draw[transform canvas={yshift=-0.3ex},->,swap](A) to node[below=3] {$Fg$} (B);
	\end{tikzpicture}
	\]
	is admissible in $D$.
\end{proof}

\begin{definition}\label{eq:8}
	Let $J:\Theta^\op\to C$ be a globular theory. We say that $C$ is an \emph{$(\infty,0)$-coherator} if $J$ is contractible and if there is a diagram of the form 
	\[
	C_0=\Theta_0^\op\to C_1\to C_2\to\cdots\to C_n\to\cdots
	\]
	where $C_n\to C_{n+1}$ is a morphism of globular theories together with a set $U_n$ of admissable pairs of $C_n$ such that $C_{n+1}$ is obtained from $C_n$ by formally adding a lift to each admissable pair in $U_n$ for $n\geq 0$ and $C$ is the colimit of the diagram. 
\end{definition}

\begin{notation}
Let $C$ be an $(\infty,0)$-coherator.  The category $\inftygpd_C=\Mod_{\Theta_0^\op}(C)$ is the full subcategory of $[C,\Set]$ containing the globular product preserving functors. The objects are called \emph{Grothendieck $\infty$-groupoids} or \emph{$\infty$-groupoids}, for short.
\end{notation} 

\section{Kelly's Small Object Argument}
We now present a factorization argument that we call \emph{Kelly's Small Object Argument}. We begin this section by providing some exposition on algebraic weak factorization systems. Given a set of maps $I$, with small domains and targets, in a category $C$ with all pullbacks, we construct the reduced factorization of a map (see \ref{RFA_construct_on_ob}). We finish this section by constructing a \emph{strong factorization system} (see \ref{sfs})  from every set of maps, with small domains and targets, in a locally finite presentable category (see \ref{choice_OFS}).  The argument that builds our strong factorization system is called Kelly's Small Object Argument.

\begin{definition}\label{AWFS} (Garner, Definition 2.2 of \cite{Garn}) Let $(\mathcal{L},\mathcal{R})$ be a wfs on a category $C$.  An \emph{algebraic realisation} of $(\mathcal{L},\mathcal{R})$ is given by the following pieces of data:
	\begin{itemize}
		\item  the wfs is given by a functorial factorisation system $(L,E,R)$;
		\item for each commutative square 
		\[
	\begin{tikzpicture}[node distance=1.5cm,auto]
		\node (A) {$u$};
		\node (B)[right of=A] {$w$};
		\node(C)[below of=A] {$v$};
		\node (D)[right of=C] {$x$};
		\draw[->] (A) to node {$h$}(B);
		\draw[->,swap] (A) to node{$f$}(C);
		\draw[->] (B) to node{$g$}(D);
		\draw[->,swap] (C) to node {$k$}(D);
	\end{tikzpicture}
	\]
		there is a choice of filler 
\[
\begin{tikzpicture}[node distance=2cm]
	\node(A){$u$};
	\node (B)[below of=A]{$Ef$};
	\node (C)[right of=A]{$w$};
	\node (D)[right of=C]{$Eg$};
	\node (E)[right of=B]{$v$};
	\node (F)[right of=E]{$x$};
	\draw[->](A) to node[left=3]{$Lf$} (B);
	\draw[->](A) to node[above=3]{$h$} (C);
	\draw[->](C) to node[above=3]{$Lg$} (D);
	\draw[->](B) to node[below=3]{$Rf$}(E);
	\draw[->](E) to node[below=3]{$k$}(F);
	\draw[->](D) to node[right=3]{$Rg$}(F);
	\draw[->](B) to node[above=3]{$E(h,k)$}(D);
\end{tikzpicture};
\]
\item for each $f: x\to y$ in $C$, there are choices of
		fillers for the following squares:
			\label{dist_law_comps}
			\[
	\begin{tikzpicture}[node distance=2cm,auto]
		\node (A) {$x$};
		\node (B)[right of=A] {$ELf$};
		\node(C)[below of=A] {$Ef$};
		\node (D)[right of=C] {$Ef$};
		\draw[->] (A) to node {$L^2f$}(B);
		\draw[->,swap] (A) to node{$Lf$}(C);
		\draw[->] (B) to node{$RLf$}(D);
		\draw[->,swap] (C) to node {$1_{Ef}$}(D);
		\draw[->,dotted](C) to node[left=3]{$\sigma_f$}(B);
	\end{tikzpicture}
	 \qquad \text{and} \qquad
	\begin{tikzpicture}[node distance=2cm,auto]
		\node (A) {$Ef$};
		\node (B)[right of=A] {$Ef$};
		\node(C)[below of=A] {$ERf$};
		\node (D)[right of=C] {$y$};
		\draw[->] (A) to node {$1_{Ef}$}(B);
		\draw[->,swap] (A) to node{$LRf$}(C);
		\draw[->] (B) to node{$Rf$}(D);
		\draw[->,swap] (C) to node {$R^2f$}(D);
		\draw[->,dotted](C) to node[left=3]{$\pi_f$}(B);
	\end{tikzpicture}
\]
\end{itemize}
	subject to the following axioms:
	\begin{itemize}
		\item the assignation $f \mapsto Lf$ extends to the functor
		part of a comonad $L$ on $\Arr(C)$ whose
		counit map at $f$ is 
		\[
		(1, Rf): Lf \to f
		\]
		and whose comultiplication $\Delta$ is 
		\[
		(1, \sigma_f):Lf \to L^2f;
		\]
		\item the assignation $f \mapsto Rf$ extends to the functor
		part of a monad $R$ on $\Arr(C)$ whose
		unit map at $f$ is $(Lf,1):f \to R f$
		and whose multiplication $\mu$ is $(\pi_f, 1):
		R^2f \to Rf$;
		\item the natural transformation (combined by putting together the lower right square of the first diagram and the upper left square in the horizontally adjacent squares right above) $\delta:LR \Rightarrow RL:
		\Arr(C) \to \Arr(C)$ whose component at $f$
		is $(\sigma_f, \pi_f):LRf \to RLf$ describes a distributive law between $L$ and $R$ and satisfies $\text{dom}\delta=\text{cod}\Delta$ and $\text{cod}\delta=\text{dom}\mu$.
	\end{itemize}
\end{definition}

\begin{definition}
Let $C$ be a category together with all the structure specified by Definition \ref{AWFS} on it. Then we call $(L,E,R,\delta)$ an \emph{algebraic weak factorization system} or \emph{AWFS}, for short.
\end{definition}

\begin{proposition}\label{universalrealisation} (Garner, Proposition 2.3 of \cite{Garn})
	Let $C$ be a category and let $I$ be a set of maps in
	$C$. Suppose that $C$ is locally finite presentable. Then the w.f.s.\ $(\mathcal{L}, \mathcal{R})$ cofibrantly generated by $I$ has a universally determined algebraic realisation $(L,E,R,\delta)$. 
\end{proposition}

\begin{definition}
We call $(L,E,R,\delta)$ the \emph{AWFS generated by $I$} and the statement is called the \emph{algebraic small object argument} or \emph{ASOA} for short.
\end{definition}

\begin{remark}\label{fac_condtns}
Let $C$ be a category and $I$ be a set of maps that satisfies the criterion of Proposition \ref{universalrealisation}. Given a morphism $f:e\to d$ in $C$, we are given
\begin{itemize}
\item a factorization 
\[
e\xrightarrow{Lf}Ef\xrightarrow{Rf}d
\]
of $f$ in $C$
\item a choice of lift $\delta_i(p,q):t(i)\to w$ for all diagrams
\[
	\begin{tikzpicture}[node distance=2cm,auto]
		\node (A) {$s(i)$};
		\node (B)[right of=A] {$w$};
		\node(C)[below of=A] {$t(i)$};
		\node (D)[right of=C] {$d$};
		\draw[->] (A) to node {$p$}(B);
		\draw[->,swap] (A) to node{$i$}(C);
		\draw[->] (B) to node{$h$}(D);
		\draw[->,swap] (C) to node {$q$}(D);
	\end{tikzpicture}
	\]
that commute where
\begin{itemize}
	\item $i\in I$
	\item  the maps
	\[
	p:s(i)\to w
	\]
	\[
	q:t(i)\to d
	\]
are morphisms in $C$.
\end{itemize}
\end{itemize}
\end{remark}

\begin{lemma}\label{univ_prop_of_AWFS_fact}
Let $C$ be a category and let $I$ be a set of maps in $C$ that satisfies the criterion of Proposition \ref{universalrealisation}. Let $(L,E,R,\delta)$ be the AWFS generated by $I$. generated by $I$. Moreover, let $f:e\to d$ be a morphism of $C$ together with a factorization 
\[
e\xrightarrow{g}w\xrightarrow{h}d
\]
of $f$ in $C$ that satisfies the criterion of Remark \ref{fac_condtns}. Then there is a unique map $g':Ef\to w$ such that

\begin{itemize}
\item the diagram
\[
\begin{tikzpicture}[node distance=2cm]
\node (A){$c$};
\node(B)[below of=A]{$Ef$};
\node (C)[below of=B]{$d$};
\node (D)[right of=B]{$w$};
\draw[->](A) to node[left=3]{$Lf$}(B);
\draw[->](B) to node[left=3]{$Rf$}(C);
\draw[->](A) to node[above=3]{$g$}(D);
\draw[->](B) to node[above=3]{$g'$}(D);
\draw[->](C) to node[below=3]{$h$}(D);
\end{tikzpicture}
\]
commutes
\item the diagram
\[
\begin{tikzpicture}[node distance=2cm]
\node(B){$t(i)$};
\node (C)[right of=B]{$Ef$};
\node (D)[below of=C]{$w$};
\draw[->](B) to node[above=3]{$\delta_i(p,q)$}(C);
\draw[->](B) to node[left=3]{$\phi_i(p,q)$}(D);
\draw[->](C) to node[right=3]{$g'$}(D);
\end{tikzpicture}
\]
commutes for $i\in I$ and pairs $(p,q)$.
\end{itemize}
\end{lemma}

\begin{proof}
Just apply the universal property of the colimits applied by Garner in $\S$ 6 of \cite{Garn3} used to construct the AWFS.
\end{proof}

\begin{construction}\label{RFA_construct_on_ob}
Let $C$ be a cocomplete category with all pullbacks and let $I$ be a set of maps in
$C$. Suppose the domains and targets of the  morphisms of $I$ are sequentially small. Let $f:e\to d$ be a morphism of $C$. Define
\begin{itemize}
\item the set $K_f$ to consist of triples $(i,p,q)$
where
\begin{itemize}
\item $i\in I$
\item $p\in C(s(i),e)$
\item $q\in C(t(i),d)$
\end{itemize}
such that
\[
	\begin{tikzpicture}[node distance=2cm,auto]
		\node (A) {$s(i)$};
		\node (B)[right of=A] {$e$};
		\node(C)[below of=A] {$t(i)$};
		\node (D)[right of=C] {$d$};
		\draw[->] (A) to node {$p$}(B);
		\draw[->,swap] (A) to node{$i$}(C);
		\draw[->] (B) to node{$f$}(D);
		\draw[->,swap] (C) to node {$q$}(D);
	\end{tikzpicture}
	\]
commutes
\item the set $\mathpzc{lift}^f(i,p,q)$ consists of all lifts 
\[
w:t(i)\to e
\]
for the diagram
\[
	\begin{tikzpicture}[node distance=2cm,auto]
		\node (A) {$s(i)$};
		\node (B)[right of=A] {$e$};
		\node(C)[below of=A] {$t(i)$};
		\node (D)[right of=C] {$d$};
		\draw[->] (A) to node {$p$}(B);
		\draw[->,swap] (A) to node{$i$}(C);
		\draw[->] (B) to node{$f$}(D);
		\draw[->,swap] (C) to node {$q$}(D);
	\end{tikzpicture}
	\]
that commutes.
\end{itemize}
Now define a category $\Lambda^f$ to have object set
\[
\ob(\Lambda^f):=K_f\cup{1}
\]
and hom-sets
\[
\Lambda^f((i,p,q),(i',p',q'))=\begin{cases}
\{1_{i,p,q}\}&\text{ if }(i,p,q)=(i',p',q')\\
\emptyset&\text{ otherwise}
\end{cases},
\]
\[
\Lambda^f(1,1)={1_1},
\]
\[
\Lambda^f(1,(i,p,q))=\emptyset
\]
for all $(i,p,q)\in K_f$,
\[
\Lambda^f((i,p,q),1)=\mathpzc{lift}^f(i,p,q)
\]
for all $(i,p,q)\in K_f$. Define $F_f:\Lambda^f\to C$ by setting
\[
F_f((i,p,q))=t(i)
\]
for all  $(i,p,q)\in K_f$,
\[
F_f(1)=e,
\]
and
\[
F_f(w)=w
\]
for all $w\in \mathpzc{lift}^f(i,p,q)$. Let 
\[
(E_+(f),\Lambda^f,F_f,\kappa:F_f\Rightarrow *_{E_+(f)})
\]
be the colimit cone of the diagram $F_f:\Lambda^f\to C$. We set $L_+(f):e\to E_+(f)$ to be 
\[
L_+(f):=\kappa_1.
\]
\end{construction}

\begin{lemma}\label{univ_prop_of_red_fact}
Let $C$ be a category and $I$ be a set of maps that satisfy the hypotheses of Construction \ref{RFA_construct_on_ob}. Let $f:e\to d$ and $\gamma:e\to e'$ be maps. Suppose that for every diagram 
\[
	\begin{tikzpicture}[node distance=2cm,auto]
		\node (A) {$s(i)$};
		\node (B)[right of=A] {$e$};
		\node(C)[below of=A] {$t(i)$};
		\node (D)[right of=C] {$d$};
		\draw[->] (A) to node {$p$}(B);
		\draw[->,swap] (A) to node{$i$}(C);
		\draw[->] (B) to node{$f$}(D);
		\draw[->,swap] (C) to node {$q$}(D);
	\end{tikzpicture}
	\]
that commutes where $i\in I$ and every pair of lifts $\delta,\delta':t(i)\rightrightarrows e$, we have that
\[
\gamma\circ \delta=\gamma\circ\delta'.
\] 
Then there is a unique map $\lambda:E_+(f)\to e'$ such that 
\[
\lambda\circ L_+(f)=\gamma.
\]
\end{lemma}

\begin{proof}
This is just the interpretation of the colimit of Construction \ref{RFA_construct_on_ob}.
\end{proof}

By the universal property of Lemma \ref{univ_prop_of_red_fact}, there is a unique map $R_+(f):E_+(f)\to d$ such that
\[
f=R_+(f)\circ L_+(f).
\]
This gives us a factorization
\[
e\xrightarrow{L_+(f)}E_+(f)\xrightarrow{R_+(f)}d
\]
of $f$ that we call the \emph{reduced factorization of $f$}. 

\begin{definition}
Let $f:x\to y$ and $g:a\to b$ be two morphisms in $C$. We say that the ordered pair $(f, g)$ is \emph{uniquely orthogonal}, denoted by $f\downarrow g$, provided that any commutative diagram
	\[
	\begin{tikzpicture}[node distance=2cm,auto]
		\node (A) {$x$};
		\node (B)[right of=A] {$a$};
		\node(C)[below of=A] {$y$};
		\node (D)[right of=C] {$b$};
		\draw[->] (A) to node {}(B);
		\draw[->,swap] (A) to node{$f$}(C);
		\draw[->] (B) to node{$g$}(D);
		\draw[->,swap] (C) to node {}(D);
	\end{tikzpicture}
	\]
	admits a unique lifting.
\end{definition}

\begin{notation}
Let $C$ be a category and $H$ be a set of maps of $C$. We write 
\[
H^{\downarrow}:=\{f\in\mor(C):h\downarrow f\text{ for all }h\in H\}
\]
and 
\[
 H^{\uparrow}:=\{f\in\mor(C):f\downarrow h\text{ for all }h\in H\}
\]
\end{notation}

\begin{definition}
Let $(\zL,\zR)$ be a weak factorization system on a category $C$. We say it is an \emph{orthogonal factorization system} (or \emph{OFS} for short) if every lifting problem has a unique solution.
\end{definition}

\begin{theorem}(Kelly, Theorem 11.3 of \cite{Kell})\label{Kelly_OFS}
Let $C$ be a category that has all pullbacks and $\kappa$ be a regular cardinal. Moreover, let $H$ be a small set of morphisms of $C$ such that the domains and targets of each morphism are $\kappa$-compact objects. Then $(H^{\uparrow\downarrow},H^{\downarrow})$ is an OFS.
\end{theorem}

\begin{remark}
Here $\kappa$-compact object refers to an object whose induced representable functor preserves $\kappa$-filtered colimits. 
\end{remark}

\begin{lemma}\label{choice_OFS}
	Let $C$ be a category and let $(\zL,\zR)$ be an OFS on $C$. Giving a choice of factorization for every map in $C$ is equivalent to upgrading the underlying weak factorization system to an AWFS where solutions to lifting problems are unique. 
\end{lemma}

\begin{proof}
This is a reinterpretation of Proposition 2.8 from Garner in \cite{Garn3}.
\end{proof}

\begin{definition}\label{sfs}
A \emph{strong factorization system} is an AWFS where solutions to lifting problems are unique.
\end{definition}

\begin{construction}\label{Kelly's_soa_funct_fact}
Let $C$ be a locally finite presentable category and $I$ be a set of morphisms whose domains and targets are sequentially small. Define a functorial factorization system 
\[(\ol{L},\ol{E},\ol{R}):\Arr(C)\to \Arr(C)\times_C\Arr(C)
\]
 on objects where
\[
\ol{E}(f)=E_+(R(f)),
\]
\[
\ol{L}(f):=L_+(R(f))\circ L(f),
\]
and
\[
\ol{R}(f)=R_+(R(f)).
\]
Suppose that
\[
	\begin{tikzpicture}[node distance=2cm,auto]
		\node (A) {$e$};
		\node (B)[right of=A] {$w$};
		\node(C)[below of=A] {$d$};
		\node (D)[right of=C] {$x$};
		\draw[->] (A) to node {$h$}(B);
		\draw[->,swap] (A) to node{$f$}(C);
		\draw[->] (B) to node{$g$}(D);
		\draw[->,swap] (C) to node {$k$}(D);
	\end{tikzpicture}
	\]
commutes. There is a choice of filler 
\[
E(h,k):Ef\to Eg
\]
such that the diagram
\[
\begin{tikzcd}
e \ar[r,"h"] \ar[dd,swap,"Lf"] & w\ar[r,"Lg"]&
Eg \ar[dd,"Rg"] \\
\\
Ef \ar[r,swap,"Rf"] \ar[uurr,"E(h{,} k)"] &d\ar[r,swap,"k"]&
				x
			\end{tikzcd}
\]
commutes. Notice that given a square 
\[
	\begin{tikzpicture}[node distance=2cm,auto]
		\node (A) {$s(i)$};
		\node (B)[right of=A] {$E(f)$};
		\node(C)[below of=A] {$t(i)$};
		\node (D)[right of=C] {$d$};
		\draw[->] (A) to node {$p$}(B);
		\draw[->,swap] (A) to node{$i$}(C);
		\draw[->] (B) to node{$R(f)$}(D);
		\draw[->,swap] (C) to node {$q$}(D);
	\end{tikzpicture}
	\]
there is an induced commutative diagram.
\[
	\begin{tikzpicture}[node distance=2cm,auto]
		\node (A) {$s(i)$};
		\node (B)[right of=A] {$E(f)$};
		\node(C)[below of=A] {$t(i)$};
		\node (D)[right of=C] {$d$};
		\node(E)[right of=B] {$E(g)$};
		\node(F)[right of=D]{$x$};
		\draw[->] (A) to node {$p$}(B);
		\draw[->,swap] (A) to node{$i$}(C);
		\draw[->] (B) to node{$R(f)$}(D);
		\draw[->,swap] (C) to node {$q$}(D);
		\draw[->](B) to node {$E(h,k)$}(E);
		\draw[->,swap](D) to node {$k$}(F);
		\draw[->](E) to node {$R(g)$}(F);
	\end{tikzpicture}
	\]
The composite diagram has a choice of lift $\gamma:t(i)\to E(g)$ by the ASOA. Notice that if $\kappa:t(i)\to E(g)$ is another lift, then 
\[
L_+(R(g))\circ \kappa=L_+(R(g))\circ \gamma.
\]
Therefore if $\delta,\delta':t(i)\to E(g)$ are lifts for the square on the left, then
\[
E(h,k)\circ\delta,E(h,k)\circ\delta'
\]
are lifts for the composite square, so that 
\[
L_+(R(g))\circ E(h,k)\circ\delta=L_+(R(g))\circ E(h,k)\circ\delta'.
\]
By the universal property given in Lemma \ref{univ_prop_of_red_fact}, there is a unique map 
\[
\ol{E}(h,k):\ol{E}(f)\to\ol{E}(g)
\]
such that the diagram
\[
	\begin{tikzpicture}[node distance=2cm,auto]
		\node (A) {$e$};
		\node (B)[right of=A] {$\ol{E}(f)$};
		\node(C)[below of=A] {$w$};
		\node (D)[right of=C] {$\ol{E}(g)$};
		\node(E)[right of=B] {$d$};
		\node(F)[right of=D]{$x$};
		\draw[->] (A) to node {$\ol{L}(f)$}(B);
		\draw[->,swap] (A) to node{$h$}(C);
		\draw[->] (B) to node{$\ol{E}(h,k)$}(D);
		\draw[->,swap] (C) to node {$\ol{L}(g)$}(D);
		\draw[->](B) to node {$\ol{R}(f)$}(E);
		\draw[->,swap](D) to node {$\ol{R}(g)$}(F);
		\draw[->](E) to node {$k$}(F);
	\end{tikzpicture}
	\]
commutes. This gives us our action on morphisms and uniqueness force functorality.
\end{construction}

\begin{remark}\label{Remark_for_KSOA}
Let $C$ be a category and $I$ be a set of morphisms of $C$ that satisfies the criterion of Construction \ref{Kelly's_soa_funct_fact}. Given a morphism $f:e\to d$ in $C$, we are given
\begin{itemize}
\item a factorization 
\[
e\xrightarrow{\ol{L}f}\ol{E}f\xrightarrow{\ol{R}f}d
\]
of $f$ in $C$
\item a unique lift $\delta_i(p,q):t(i)\to \ol{E}(f)$ for all diagrams 
\[
	\begin{tikzpicture}[node distance=2cm,auto]
		\node (A) {$s(i)$};
		\node (B)[right of=A] {$\ol{E}(f)$};
		\node(C)[below of=A] {$t(i)$};
		\node (D)[right of=C] {$d$};
		\draw[->] (A) to node {$p$}(B);
		\draw[->,swap] (A) to node{$i$}(C);
		\draw[->] (B) to node{$\ol{R}(f)$}(D);
		\draw[->,swap] (C) to node {$q$}(D);
	\end{tikzpicture}
	\]
that commute where
\begin{itemize}
	\item $i\in I$
	\item the maps
	\[
	p:s(i)\to w
	\]
	\[
	q:t(i)\to d
	\]
is a pair of morphisms in $C$.
\end{itemize}
\end{itemize}
\end{remark}

\begin{lemma}\label{Uni_prop_of_Kell's_soa}
Let $C$ be a category and $I$ be a set of maps that satisfy the criterion of Construction \ref{Kelly's_soa_funct_fact}. Let $f:e\to d$ be a morphism of $C$ together with a factorization 
\[
e\xrightarrow{g}w\xrightarrow{h}d
\]
of $f$ in $C$ that satisfies the criterion of Remark \ref{Remark_for_KSOA}. Then there is a unique map $g':\ol{E}(f)\to w$ such that
\begin{itemize}
\item the diagram
\[
\begin{tikzpicture}[node distance=2cm]
\node (A){$c$};
\node(B)[below of=A]{$\ol{E}f$};
\node (C)[below of=B]{$d$};
\node (D)[right of=B]{$w$};
\draw[->](A) to node[left=3]{$\ol{L}f$}(B);
\draw[->](B) to node[left=3]{$\ol{R}f$}(C);
\draw[->](A) to node[above=3]{$g$}(D);
\draw[->](B) to node[above=3]{$g'$}(D);
\draw[->](C) to node[below=3]{$h$}(D);
\end{tikzpicture}
\]
commutes
\item the diagram
\[
\begin{tikzpicture}[node distance=2cm]
\node(B){$t(i)$};
\node (C)[right of=B]{$\ol{E}f$};
\node (D)[below of=C]{$w$};
\draw[->](B) to node[above=3]{$\delta_i(p,q)$}(C);
\draw[->](B) to node[left=3]{$\phi_i(p,q)$}(D);
\draw[->](C) to node[right=3]{$g'$}(D);
\end{tikzpicture}
\]
commutes for $i\in I$ and pairs $(p,q)$.
\end{itemize}
\end{lemma}

\begin{proof}
By Lemma \ref{univ_prop_of_AWFS_fact}, there is a unique map $\hat{g}:Ef\to w$ that satisfies the criteria of Lemma \ref{univ_prop_of_AWFS_fact}.
Now applying the universal property of the colimit of Construction \ref{RFA_construct_on_ob}, there is a unique map $g':\ol{E}(f)\to w$ such that
\begin{itemize}
\item the diagram
\[
\begin{tikzpicture}[node distance=2cm]
\node (A){$c$};
\node(B)[below of=A]{$\ol{E}f$};
\node (C)[below of=B]{$d$};
\node (D)[right of=B]{$w$};
\draw[->](A) to node[left=3]{$\ol{L}f$}(B);
\draw[->](B) to node[left=3]{$\ol{R}f$}(C);
\draw[->](A) to node[above=3]{$g$}(D);
\draw[->](B) to node[above=3]{$g'$}(D);
\draw[->](C) to node[below=3]{$h$}(D);
\end{tikzpicture}
\]
commutes
\item the diagram
\[
\begin{tikzpicture}[node distance=2cm]
\node(B){$t(i)$};
\node (C)[right of=B]{$\ol{E}f$};
\node (D)[below of=C]{$w$};
\draw[->](B) to node[above=3]{$\delta_i(p,q)$}(C);
\draw[->](B) to node[left=3]{$\phi_i(p,q)$}(D);
\draw[->](C) to node[right=3]{$g'$}(D);
\end{tikzpicture}
\]
commutes for $i\in I$ and pairs $(p,q)$.
\end{itemize}
\end{proof}

\begin{theorem}(Kelly's Small Object Argument)\label{Kelly_SOA}
	Let $C$ be a locally finite presentable category and $I$ be a set of morphisms whose domains and targets are sequentially small. Then Construction \ref{Kelly's_soa_funct_fact} extends to a strong factorization system.
\end{theorem}
\begin{proof}
It follows from Theorem \ref{Kelly_OFS} that 
$(I^{\uparrow\downarrow},I^{\downarrow})$ is an OFS. We have that 
\[
\ol{L}(f)\in  I^{\uparrow\downarrow}
\]
 for all $f\in \mor(C)$. We only need to check that $\ol{R}(f)\in  I^{\uparrow}$  for all $f\in \mor(C)$. This just follows from Remark \ref{Remark_for_KSOA}. We are now finished with our proof. Therefore we have a strong factorization system.
\end{proof}

\section{Distributive Series}
We now provide background and extend the theory for distributive series of monads. The constructions later in this paper fall into distributive series (see Theorem \ref{distr_ser_AWFS} and Theorem \ref{distr_ser_OFS}). Moreover, we give an appendix dedicated to distributive series of monads whose monads are epi-monads (see Definition \ref{epi_monads}).
\begin{definition}(Cheng, Definition 2.3 of \cite{EC1})
	Fix $n\geq 3$. A \emph{distributive series of $n$ monads} consists of monads $T_0,T_1,\cdots,T_n$ together with a distributive law $\lambda^{i,j}:T_iT_j\Rightarrow T_jT_i$
	for $n>j>i\geq 0$ such that the diagram 
	\[
\begin{tikzpicture}[node distance=1.5cm]
	\node (A){$T_iT_jT_k$};
	\node (B)[right of=A]{};
	\node (C)[above of =B]{$T_jT_iT_k$};
	\node (D)[below of =B]{$T_iT_kT_j$};
	\node (E)[right of=C,node distance=3cm]{$T_jT_kT_i$};
	\node (F)[right of=D,node distance=3cm]{$T_kT_iT_j$};
	\node (G)[right of=B]{};
	\node (H)[right of =G,node distance=3cm]{$T_kT_jT_i$};
	\draw[->](A) to node[left=3]{$\lambda^{i,j}T_k$}(C);
	\draw[->,swap](A) to node[left=3]{$T_i\lambda^{j,k}$}(D);
	\draw[->](C)to node[above=3]{$T_j\lambda^{i,k}$}(E);
	\draw[->,swap](D) to node[below=3]{$\lambda^{i,k}T_j$}(F);
	\draw[->](E) to node[right=3]{$\lambda^{j,k}T_i$}(H);
	\draw[->,swap](F) to node[right=3]{$T_k\lambda^{i,j}$}(H);
\end{tikzpicture}
\]
commutes for $n\geq k>j>i\geq 0$. 
\end{definition}

\begin{theorem}(Cheng, Theorem 2.1 of \cite{EC1})
	Fix $n\geq 3$. Let $((T_i)_{i=0}^n,(\lambda_{i,j})_{n\geq i>j})$ be a distributive series of $n$ monads. Then for all $1\leq i < n$  we have induced monads 
	\[T_i T_{i-1} \cdots T_0 \quad \mbox{and} \quad T_{n} T_{n-1} \cdots T_{i+1}\]
	
	\noindent together with a distributive law of \ $T_i T_{i-1} \cdots T_0$\  over \ $T_{n} T_{n-1} \cdots T_{i+1}$\ i.e. 
	\[(T_i T_{i-1} \cdots T_0) (T_{n} T_{n-1} \cdots T_{i+1}) \Rightarrow (T_{n} T_{n-1} \cdots T_{i+1})(T_i T_{i-1} \cdots T_0)\]
	
	\noindent given by the composites of the $\lambda^{ij}$.  Moreover, all the monad structures on \ $T_n T_{n-1} \cdots T_0$ induced by combining the structure maps of the monads and distributive laws are the same.
\end{theorem}

We now extend the definition of Cheng to where we have an $\bN$-indexed collection of monads. 

\begin{definition}
	A \emph{distributive series of monads} consists of monads $T_0,T_1,\cdots,T_n,\cdots$ together with a distributive law $\lambda^{i,j}:T_iT_j\Rightarrow T_jT_i$
	for $j>i\geq 0$ such that the diagram 
	\[
\begin{tikzpicture}[node distance=1.5cm]
	\node (A){$T_iT_jT_k$};
	\node (B)[right of=A]{};
	\node (C)[above of =B]{$T_jT_iT_k$};
	\node (D)[below of =B]{$T_iT_kT_j$};
	\node (E)[right of=C,node distance=3cm]{$T_jT_kT_i$};
	\node (F)[right of=D,node distance=3cm]{$T_kT_iT_j$};
	\node (G)[right of=B]{};
	\node (H)[right of =G,node distance=3cm]{$T_kT_jT_i$};
	\draw[->](A) to node[left=3]{$\lambda^{i,j}T_k$}(C);
	\draw[->,swap](A) to node[left=3]{$T_i\lambda^{j,k}$}(D);
	\draw[->](C)to node[above=3]{$T_j\lambda^{i,k}$}(E);
	\draw[->,swap](D) to node[below=3]{$\lambda^{i,k}T_j$}(F);
	\draw[->](E) to node[right=3]{$\lambda^{j,k}T_i$}(H);
	\draw[->,swap](F) to node[right=3]{$T_k\lambda^{i,j}$}(H);
\end{tikzpicture}
\]
commutes for $k>j>i\geq 0$. 
\end{definition}

 \begin{notation}
 We write 
	\[
	T:=((T_i)_{i=0}^\infty,(\lambda_{i,j})_{i>j})
	\]
	to denote a distributive series of monads. 
\end{notation}

\begin{corollary}\label{thmChengReIm_corr}
	Let $((T_i)_{i=0}^\infty,(\lambda_{i,j})_{i>j})$ be a distributive series of monads. Then for all $n\geq 2$ and $1\leq i < n$  we have induced monads 
	\[T_i T_{i-1} \cdots T_0 \quad \mbox{and} \quad T_{n} T_{n-1} \cdots T_{i+1}\]
	
	\noindent together with a distributive law of \ $T_i T_{i-1} \cdots T_0$\  over \ $T_{n} T_{n-1} \cdots T_{i+1}$\ i.e. 
	\[(T_i T_{i-1} \cdots T_0) (T_{n} T_{n-1} \cdots T_{i+1}) \Rightarrow (T_{n} T_{n-1} \cdots T_{i+1})(T_i T_{i-1} \cdots T_0)\]
	
	\noindent given by the composites of the $\lambda^{ij}$.  Moreover, all the monad structures on \ $T_n T_{n-1} \cdots T_0$ induced by combining the structure maps of the monads and distributive laws are the same for $n\geq 2$.
\end{corollary}

\begin{proof}
The proof of this is identical to the one Cheng provided for Theorem 2.1 of \cite{EC1}.
\end{proof}

\begin{definition}
A map $f:T\to S$ of distributive series consists of a monad morphism 
\[
f_i:T_i\to S_i
\]
for all $i\geq 0$ such that the square 
\[
\begin{tikzpicture}[node distance=2cm,auto]
	\node (A) {$T_iT_j$};
	\node (B)[right of=A] {$T_jT_i$};
	\node(C)[below of=A] {$S_iS_j$};
	\node (D)[right of=C] {$S_jS_i$};
	\draw[->] (A) to node {$\lambda_{i{,}j}$}(B);
	\draw[->,swap] (A) to node{$f_if_j$}(C);
	\draw[->] (B) to node{$f_jf_i$}(D);
	\draw[->,swap] (C) to node {$\lambda'_{i,j}$}(D);
\end{tikzpicture}
\]
commutes for $i>j\geq 0$.
\end{definition}

\begin{notation}
We use $\DistSys$ to denote the category of distributive systems and maps between them.
\end{notation}

\begin{definition}
	Let $n\geq 0$, $C$ be a category and 
	\[
	T:=((T_i)_{i=0}^\infty,(\lambda_{i,j})_{i>j})
	\]
	be a distributive series of monads. The \emph{$n$th associated monad} is the monad structure on $T_nT_{n-1}\cdots T_0$ induced in Theorem \ref{thmChengReIm_corr} and we denote it by $(\hat{T}_n,\hat{\eta}_n,\hat{\mu}_n)$.
\end{definition}

\begin{definition}\label{ind_funct_on_dist}
	Let $C$ be a category and 
	\[
	T:=((T_i)_{i=0}^\infty,(\lambda_{i,j})_{i>j})
	\]
	be a distributive series of monads. The \emph{associated pointed endofunctor of the distributive series}  $\hat{T}$, if it exists, is defined to be
	to be the colimit of the diagram
	\[
	1_{C}\xrightarrow{\eta^0} T_0\xrightarrow{\eta^1_{T_0}} T_1T_0\xrightarrow{\eta^2_{T_1T_0}}\cdots
	\] \label{diag_for_ass_pntd_endo}
	in the category of endofunctors on $C$ with pointing given to be the map
	\[
	\hat{\eta}:1_C\Rightarrow \hat{T}
	\]
	induced by taking colimits. 
\end{definition}
We can not always upgrade the pointed endofunctor $(\hat{T},\hat{\eta})$, if it exists, to a monad and we give a special name to the distributive series when it can.  
	\begin{definition}\label{n-tail}
	Let $n\geq 0$, $C$ be a category and 
	\[
	T:=((T_i)_{i=0}^\infty,(\lambda_{i,j})_{i>j})
	\]
	be a distributive series of monads. The \emph{$n$-tail} of the associated pointed endofunctor of the distributive series, if it exists, is the unique structure map 
	\[
	\kappa_n:\hat{T}_n\to\hat{T}
	\]
	associated to the colimit used to form the associated pointed endofunctor of the distributive series.
	\end{definition}
Here the $n$-tail is the colimit to the diagram in Definition \ref{diag_for_ass_pntd_endo} up to and including $\hat{T}_n=T_nT_{n-1}\cdots T_0$ for $n\geq 0$.	
	\begin{definition}\label{compl_dist_series}
	We say that a distributive series of monads is \emph{completable} if its associated pointed endofunctor $(\hat{T},\hat{\eta})$, if it exists, extends to a monad $(\hat{T},\hat{\eta},\hat{\mu})$ such that the diagram 
	\[
\begin{tikzpicture}[node distance=2cm,auto]
	\node (A) {$\hat{T}_n^2$};
	\node (B)[right of=A] {$\hat{T}^2$};
	\node(C)[below of=A] {$\hat{T}_n$};
	\node (D)[right of=C] {$\hat{T}$};
	\draw[->] (A) to node {$\kappa_n\kappa_n$}(B);
	\draw[->,swap] (A) to node{$\hat{\mu}^n$}(C);
	\draw[->] (B) to node{$\hat{\mu}$}(D);
	\draw[->,swap] (C) to node {$\kappa_n$}(D);
\end{tikzpicture}
\]
commutes for $n\geq 0$.
\end{definition}

\section{A Special Choice of Coherator}
In this section, we begin by constructing a cofibrantly generated algebraic weak factorization system on $\Th_{\Theta_0^\op}$. We take the fibrant replacement of the identity theory to obtain a theory which served as the initial choice of coherator we were going to work with. Then, we partition the generating maps for the cofibrantly generated algebraic weak factorization system and obtain an $\aleph_0$-indexed collection algebraic weak factorization systems. We prove that the induced fibrant replacement monads can be upgraded to a distributive series of monads. We finish by showing that the distributive series is completable in the sense of  \ref{compl_dist_series}.
\begin{definition}\label{spheres_in_inf_gpds}
	For all $\vec{p}\in\ob(\Theta_0^\op)$ and $k\geq 0$ with $\hgt(\vec{p})\leq k+1$ , define 
	\[
	S_{\vec{p},k}
	\]
	to be the theory obtained by freely adding two maps 
\[
		\begin{tikzpicture}[node distance=2cm]
			\node (A) {$\vec{p}$};
			\node (B)[right of=A]{$k$};
			\draw[transform canvas={yshift=0.3ex},->] (A) to node[above=3] {$f$} (B);
			\draw[transform canvas={yshift=-0.3ex},->,swap](A) to node[below=3] {$g$} (B);
		\end{tikzpicture}
		\]
to $\Theta_0^\op$ such that $s\circ f=s\circ g$ and $t\circ f=t\circ g$ when $k\geq 1$. This comes equipped with a theory structure map.
	\[
	\Theta_0^\op\hookrightarrow S_{\vec{p},k}
	\]
\end{definition}

\begin{definition}\label{disks_in_inf_gpds}
Let $D_{\vec{p},k}$ be obtained by freely adding a map $\delta_{f,g}:\vec{p}\to k+1$ to $S_{\vec{p},k}$  such that the triangles 

\[
		\begin{tikzpicture}[node distance=2cm]
			\node (A) {$\vec{p}$};
			\node (B)[right of=A]{$k$};
			\node (C) [above of=B]{$k+1$};
			\draw[transform canvas={yshift=0.3ex},->] (A) to node[above=3] {$f$} (B);
			\draw[transform canvas={yshift=-0.3ex},->,swap](A) to node[below=3] {$g$} (B);
			\draw[transform canvas={xshift=-0.3ex},->] (C) to node[left=3] {$s$} (B);
			\draw[transform canvas={xshift=0.3ex},->,swap](C) to node[right=3] {$t$} (B);
			\draw[->](A) to node[above=4] {$\delta_{f,g}$}(C);
		\end{tikzpicture}
		\]
commute. This comes equipped with the following theory structure map.
\[
\Theta_0^\op\hookrightarrow D_{\vec{p},k}
\]
\end{definition}
 
\begin{notation}\label{incl_of_sphere_into_disks}
There is naturally an inclusion map of theories.
\[
j_{\vec{p},k}:S_{\vec{p},k}\to D_{\vec{p},k}
\] 
for all tables of dimensions $\vec{p}$ and $n\geq 0$. We let 
	\[
	I:=\{S_{\vec{p},k}\xrightarrow{j_{\vec{p},k}} D_{\vec{p},k}:\vec{p}\in\ob(C),k\geq 0\}.
	\]
\end{notation}
\begin{lemma} \label{admiss_1}
	The objects constructed in Definition \ref{spheres_in_inf_gpds} and Definition \ref{disks_in_inf_gpds}are presentable objects of $\Th_{\Theta_0^\op}$.  Moreover, the set $I$ of Notation \ref{incl_of_sphere_into_disks} is admissible for the ASOA of \cite{Garn2}.
\end{lemma}

\begin{proof}
See Subsection 3.11 and Lemma 3.12 of \cite{Malt}.
\end{proof}

\begin{notation}
	Let $(L_I,E_I,R_I,\delta_I)$ be the AWFS cofibrantly generated by $I$ in $\Th_{\Theta_0^\op}$. 
\end{notation}
\begin{lemma}
	The fibrant replacement of the initial object $\id_{\Theta_0^\op}$, denote it by 
	\[
	J^{FC}:\Theta_0^\op\to FC,
	\]
	is an $(\infty,0)$-coherator.
\end{lemma}

\begin{proof}
Just rework the proof of Theorem 3.14 of \cite{Malt}.
\end{proof}

\begin{remark}
We use $FC$ to denote that $FC$ was our first idea of a special choice of coherator for weak $\infty$-groupoids. We note that $FC$ is the reduced coherator mentioned in Example 2.12 of \cite{Dim1}.  However, we shall not $FC$ as our choice of theory. 
\end{remark}

	We obtain a diagram
	\[
	\begin{tikzcd}
		\vdots\ar[d,shift left=.75ex,"t^*"]
		\ar[d,shift right=.75ex,swap,"s^*"]&\vdots\ar[d,shift left=.75ex,"t^*"]
		\ar[d,shift right=.75ex,swap,"s^*"]&\vdots\ar[d,shift left=.75ex,"t^*"]
		\ar[d,shift right=.75ex,swap,"s^*"]&\cdots&\vdots\ar[d,shift left=.75ex,"t^*"]
		\ar[d,shift right=.75ex,swap,"s^*"]\\
		\Theta_0^\op/3\ar[r,"J^{FC}_1/3"]\ar[d,shift left=.75ex,"t^*"]
		\ar[d,shift right=.75ex,swap,"s^*"]&FC_1/3\ar[r,"J^{FC}_2/3"]\ar[d,shift left=.75ex,"t^*"]
		\ar[d,shift right=.75ex,swap,"s^*"]&FC_2/3\ar[r,"J^{FC}_3/3"]\ar[d,shift left=.75ex,"t^*"]
		\ar[d,shift right=.75ex,swap,"s^*"]&\cdots&\cong FC/3\ar[d,shift left=.75ex,"t^*"]
		\ar[d,shift right=.75ex,swap,"s^*"]\\
		\Theta_0^\op/2\ar[r,"J^{FC}_1/2"]\ar[d,shift left=.75ex,"t^*"]
		\ar[d,shift right=.75ex,swap,"s^*"]&FC_1/2\ar[r,"J^{FC}_2/2"]\ar[d,shift left=.75ex,"t^*"]
		\ar[d,shift right=.75ex,swap,"s^*"]&FC_2/2\ar[r,"J^{FC}_3/2"]\ar[d,shift left=.75ex,"t^*"]
		\ar[d,shift right=.75ex,swap,"s^*"]&\cdots&\cong FC/2\ar[d,shift left=.75ex,"t^*"]
		\ar[d,shift right=.75ex,swap,"s^*"]\\
		\Theta_0^\op/1\ar[r,"J^{FC}_1/1"]\ar[d,shift left=.75ex,"t^*"]
		\ar[d,shift right=.75ex,swap,"s^*"]&FC_1/1\ar[r,"J^{FC}_2/1"]\ar[d,shift left=.75ex,"t^*"]
		\ar[d,shift right=.75ex,swap,"s^*"]&FC_2/1\ar[r,"J^{FC}_3/1"]\ar[d,shift left=.75ex,"t^*"]
		\ar[d,shift right=.75ex,swap,"s^*"]&\cdots&\cong FC/1\ar[d,shift left=.75ex,"t^*"]
		\ar[d,shift right=.75ex,swap,"s^*"]\\
		\Theta_0^\op/0\ar[r,"J^{FC}_1/0"]&FC_1/0\ar[r,"J^{FC}_2/0"]&FC_2/0\ar[r,"J^{FC}_3/0"]&\cdots&\cong FC/0
	\end{tikzcd}
	\]
	on coslices such that 
	\[
	s^*\circ J^{FC}_k/n+1=J^{FC}_k/n\circ s^*
	\]
	\[
	t^*\circ J^{FC}_k/n+1=J^{FC}_k/n\circ t^*
	\]
	for all $n\geq 0$ and $k\geq 1$ and $J^{FC}_k/n$ is the identity-on-objects for whenever $k\geq n+1$. This suggests that there is a more inductive way to obtain an $(\infty,0)$-coherator which mimics the classic procedure of weak enrichment. Partition $I$ as 
\[
(I_k)_{k=0}^\infty=(\{S_{\vec{p},k}\xrightarrow{j_{\vec{p},k}} D_{\vec{p},k}:\vec{p}\in\ob(\Theta_0^\op)\})_{k=0}^\infty.
\]

\begin{lemma}
The sets $I_k$ are admissible for the ASOA of \cite{Garn2}.
\end{lemma}
 \begin{proof}
See Lemma \ref{admiss_1}.
\end{proof}

\begin{notation}
	Let $(L_{I_k},E_{I_k},R_{I_k},\delta_{I_k})$ be the AWFS cofibrantly generated by $I_k$ in $\Th_{\Theta_0^\op}$ for all $k\geq 0$. Moreover, let $(R_{k})_{k=0}^\infty$ denote the corresponding fibrant replacement monads on $\Th_{\Theta_0^\op}$ by the AWFSs.
\end{notation}

\noindent We now name and organize the structure data.
\begin{notation}\label{fibr_repl_AWFS}
	Let $(R_{k})_{k=0}^\infty$ be the corresponding fibrant replacement monads on $\Th_{\Theta_0^\op}$ cofibrantly generated by the AWFSs. We write the following notation:
	\begin{itemize}
		\item $\eta^{k}:1_{\Th_{\Theta_0^\op}}\Rightarrow R_{k}$ and $\mu^{k}:R_{k}^2\Rightarrow R_{k}$ will notate the unit and multiplication, respectively, corresponding to the monad $R_{k}$ for all $k\geq 0$.
	\end{itemize}
\end{notation}

\noindent We need the following lemma.

\begin{lemma}\label{univ_prop_monad_unit}
	Given a theory $C$, a map $F:C\to D$ of globular theories, and given a choice of lift $\delta_{Ff,Fg}:\vec{p}\to k+1$ in $D$ for the image under $F$ of every admissible pair of the form
	\[
		\begin{tikzpicture}[node distance=2cm]
		\node (A) {$\vec{p}$};
			\node (B)[right of=A]{$k$};
	\draw[transform canvas={yshift=0.3ex},->] (A) to node[above=3] {$f$} (B);
			\draw[transform canvas={yshift=-0.3ex},->,swap](A) to node[below=3] {$g$} (B);
	\end{tikzpicture}
	\]
	in $C$, there is a unique map
	\[
	F':R_kC\to D
	\]
	such that
	\[
	F'\circ\eta^k_C=F
	\]
	and
	\[
	F'(\delta_{f,g})=\delta_{Ff,Fg}
	\]
	for every admissible pair of the form
\[
		\begin{tikzpicture}[node distance=2cm]
		\node (A) {$\vec{p}$};
			\node (B)[right of=A]{$k$};
	\draw[transform canvas={yshift=0.3ex},->] (A) to node[above=3] {$f$} (B);
			\draw[transform canvas={yshift=-0.3ex},->,swap](A) to node[below=3] {$g$} (B);
	\end{tikzpicture}
	\]
	in $C$.
\end{lemma}
\begin{proof}
 This is a  a reinterpretation of Lemma \ref{univ_prop_of_AWFS_fact} in this setting.
\end{proof}

\begin{remark}
This universal property actually determines the entire monad structure.
\end{remark}

\begin{definition}
Let $k>0$. We say a map $F:C\to D$ of globular theories is \emph{fully faithful below dimension k} if for all $i<k$ and tables of dimensions $\vec{p}$ with $\hgt(\vec{p})\leq i+1$, the map
\[
F_{\vec{p},k}:C(\vec{p},i)\to D(\vec{p},i)
\]
is a bijection.
\end{definition}
The following lemma is true.
\begin{lemma}
The map of theories $\eta^k_C:C\to R_k(C)$ is fully faithful below dimension $k+1$ for all theories $C$.
\end{lemma}

\noindent We now show how the fibrant replacement monads induced by the partition forms a completable distributive series of monads through a sequence of lemmas, definition and a theorem. 
\begin{definition} \label{distlaw_comp}
	Given a globular theory $C$, we define a map
	\[
	\lambda^{i,j}_C: R_iR_jC\Rightarrow R_jR_iC
	\]
	
	for $0\leq i<j$ as follows:
	\begin{itemize}
		\item we consider the diagram 
		\[
		\begin{tikzpicture}[node distance=2cm,auto]
			\node (A) {$R_jC$};
			\node (B)[right of=A]{$R_jR_iC$};
			\node (C)[below of=A]{$R_iR_jC$};
			\draw[->,swap] (A) to node {$\eta^i_{R_jC}$} (C);
			\draw[->](A) to node {$R_j\eta^i_C$}(B);
		\end{tikzpicture}
		\]
		\item every admissible pair of the form 
		\[
		\begin{tikzpicture}[node distance=2cm]
		\node (A) {$\vec{p}$};
			\node (B)[right of=A]{$i$};
	\draw[transform canvas={yshift=0.3ex},->] (A) to node[above=3] {$f$} (B);
			\draw[transform canvas={yshift=-0.3ex},->,swap](A) to node[below=3] {$g$} (B);
	\end{tikzpicture}
	\]
	is uniquely mapped onto by an admissible pair in $C$ since $\eta^j_C$ is fully faithful below dimension $i+1$
		
		\item let $\lambda^{i,j}_C:R_iR_jC\to R_jR_iC$ be the unique map such that the diagram
		\[
			\begin{tikzpicture}[node distance=2cm,auto]
				\node (A) {$R_jC$};
				\node (B)[right of=A]{$R_jR_iC$};
				\node (C)[below of=A]{$R_iR_jC$};
				\draw[->,swap] (A) to node {$\eta^i_{R_jC}$} (C);
				\draw[->](A) to node {$R_j\eta^i_C$}(B);
				\draw[->,swap] (C) to node {$\lambda^{i,j}_C$}(B);
			\end{tikzpicture}
		\]
		commutes and 
		\[
		\lambda^{i,j}_C(\delta_{f,g})=\delta_{f,g}
		\]
	where the first $\delta_{f,g}$ is the choice of lift in $R_iR_j(C)$ and the second one is the choice of lift in $R_jR_i(C)$.
	\end{itemize}
\end{definition}

\begin{remark}
From now on, given $i\geq 0$ and an admissible pair of the form 
		\[
		\begin{tikzpicture}[node distance=2cm]
		\node (A) {$\vec{p}$};
			\node (B)[right of=A]{$i$};
	\draw[transform canvas={yshift=0.3ex},->] (A) to node[above=3] {$f$} (B);
			\draw[transform canvas={yshift=-0.3ex},->,swap](A) to node[below=3] {$g$} (B);
	\end{tikzpicture}
	\]
in a theory $C$, we write $\delta_{f,g}$ to be the choice of lift added along the extension $\eta^i_C:C\to R_iC$.
\end{remark}

\begin{lemma}\label{distr_tringle_2}
	For $0\leq i<j$ and every theory $C$, the triangle 
	\[
	\begin{tikzpicture}[node distance=2cm]
		\node (B){$R_iC$};
		\node (C)[below of=B]{};
		\node (D)[left of=C]{$R_iR_jC$};
		\node (E)[right of=C]{$R_jR_iC$};
		\draw[->](B) to node[left=3]{$R_i\eta^j_C$}(D);
		\draw[->](B) to node[right=3]{$\eta^j_{R_iC}$}(E);
		\draw[->](D) to node[below=3]{$\lambda^{i,j}_C$}(E);
	\end{tikzpicture}
	\]
	commutes.
\end{lemma}
\begin{proof}
	Notice that the parts of the diagram
	\[
	\begin{tikzpicture}[node distance=2cm]
		\node (A){$C$};
		\node (B)[node distance=5cm,left of=A]{$R_iC$};
		\node (C)[below of=B]{};
		\node (D)[left of=C]{$R_iR_jC$};
		\node (E)[right of=C]{$R_jR_iC$};
		\node (F)[below of=C]{$R_jC$};
		\node (G)[node distance=1.5cm,above of=F]{$=$};
		\node (H)[node distance=1cm,right of=E]{$=$};
		\draw[->](A) to node[above=3]{$\eta^i_C$} (B);
		\draw[->,bend left=60](A) to node[right=3]{$\eta^j_C$}(F);
		\draw[->](B) to node[left=3]{$R_i\eta^j_C$}(D);
		\draw[->](B) to node[right=3]{$\eta^j_{R_iC}$}(E);
		\draw[->](D) to node[above=3]{$\lambda^{i,j}_C$}(E);
		\draw[->](F) to node[left=3]{$\eta^i_{R_jC}$}(D);
		\draw[->](F) to node[right=3]{$R_j\eta^i_C$}(E);
	\end{tikzpicture}
	\]
	with equal signs in the middle commute and the outside of the diagram commutes as well by definition and naturality. This means that
	\[
	\eta^j_{R_iC}\circ \eta^i_C=R_j\eta^i_C\circ \eta^j_C
	\]
	\[
	=\lambda^{i,j}_C\circ\eta^i_{R_jC}\circ\eta^j_C=\lambda^{i,j}_C\circ R_i\eta^j_C\circ \eta^i_C.
	\]
Given an admissible pair
\[
		\begin{tikzpicture}[node distance=2cm]
		\node (A) {$\vec{p}$};
			\node (B)[right of=A]{$j$};
	\draw[transform canvas={yshift=0.3ex},->] (A) to node[above=3] {$f$} (B);
			\draw[transform canvas={yshift=-0.3ex},->,swap](A) to node[below=3] {$g$} (B);
	\end{tikzpicture}
	\]
in $R_i(C)$, we are forced to have that
\[
\lambda^{i,j}_C(R_i(\eta^j_C)(\delta_{f,g}))=\lambda^{i,j}_C(\delta_{f,g})=\delta_{f,g}=\eta^j_{R_i(C)}(\delta_{f,g}).
\]
 By the universal property of Lemma \ref{univ_prop_monad_unit}, the triangle without an equal sign must commute. 
\end{proof}

\begin{lemma}\label{dist_law_ij}
	The maps $\lambda^{i,j}_C:R_iR_jC\to R_jR_iC$ of Lemma \ref{distlaw_comp} form a natural transformation
	\[
	\lambda^{i,j}: R_iR_j\Rightarrow R_jR_i
	\]
	for $0\leq i<j$.
\end{lemma}

\begin{proof}
	Fix $0\leq i<j$ and let $F:C\to D$ be a map. Notice that 
	\[
	R_jR_iF\circ\lambda^{i,j}_C\circ \eta^i_{R_jC}\circ \eta^j_C=	R_jR_iF\circ\lambda^{i,j}\circ R_i\eta^j_C\circ \eta^i_C
	\]
	\[
	=R_jR_iF\circ \eta^j_{R_iC} \circ \eta^i_C=\eta^j_{R_iD}\circ R_iF\circ \eta^i_C
	\]
	\[
	=\lambda^{i,j}_D\circ R_i\eta^j_D\circ R_iF\circ \eta^i_C=\lambda^{i,j}_D\circ R_iR_jF\circ R_i\eta^j_C \circ \eta^i_C
	\]
	\[
	=\lambda^{i,j}_D\circ R_iR_jF\circ \eta^i_{R_jC} \circ \eta^i_C.
	\]
	\noindent Given an admissible pair 
	\[
		\begin{tikzpicture}[node distance=2cm]
		\node (A) {$\vec{p}$};
			\node (B)[right of=A]{$j$};
	\draw[transform canvas={yshift=0.3ex},->] (A) to node[above=3] {$f$} (B);
			\draw[transform canvas={yshift=-0.3ex},->,swap](A) to node[below=3] {$g$} (B);
	\end{tikzpicture}
	\]
	in $C$, we are forced to have 
	\[
	(R_jR_iF\circ\lambda^{i,j}_C\circ \eta^i_{R_jC})(\delta_{f,g})=(\lambda^{i,j}_D\circ R_iR_jF\circ \eta^i_{R_jC})(\delta_{f,g}).
	\]
	
	Upon application of the universal property of Lemma \ref{univ_prop_monad_unit}, we must have that
	\[
	R_jR_iF\circ\lambda^{i,j}_C\circ \eta^i_{R_jC}=\lambda^{i,j}_D\circ R_iR_jF\circ \eta^i_{R_jC}.
	\]
Given an admissible pair 
	\[
		\begin{tikzpicture}[node distance=2cm]
		\node (A) {$\vec{p}$};
			\node (B)[right of=A]{$i$};
	\draw[transform canvas={yshift=0.3ex},->] (A) to node[above=3] {$f$} (B);
			\draw[transform canvas={yshift=-0.3ex},->,swap](A) to node[below=3] {$g$} (B);
	\end{tikzpicture}
	\]
	in $R_j(C)$, diagram chasing where choices are sent forces us to have that
	\[
	(R_jR_iF\circ\lambda^{i,j}_C)(\delta_{f,g})=(\lambda^{i,j}_D\circ R_iR_jF)(\delta_{f,g}).
	\]
	By the universal property of Lemma \ref{univ_prop_monad_unit}, the diagram
	\[
	\begin{tikzpicture}[node distance=2cm,auto]
		\node (A) {$R_iR_jC$};
		\node (B)[right of=A] {$R_jR_iC$};
		\node(C)[below of=A] {$R_iR_jD$};
		\node (D)[right of=C] {$R_jR_iD$};
		\draw[->] (A) to node {$\lambda^{i,j}_C$}(B);
		\draw[->,swap] (A) to node{$R_iR_jF$}(C);
		\draw[->] (B) to node{$R_jR_iF$}(D);
		\draw[->,swap] (C) to node {$\lambda^{i,j}_D$}(D);
	\end{tikzpicture}
	\]
must commute. Therefore, $\lambda^{i,j}$ is a natural transformation for $0\leq i<j$.
\end{proof}

Before we continue onto the next lemma, we remark that all the diagrams that have commuted thus far have commuted because both ways of sending a choice of lift around the appropriate diagrams are the same. This is merely a consequence of the construction of the units of the monads and the rest of the structure being generated by the units. The diagrams going forward will commute for the exact same reason. 

\begin{lemma}\label{distr_ser_fibr_repl}
	For $0\leq i<j$, $\lambda^{i,j}:R_iR_j\Rightarrow R_jR_i$ is a distributive law.	
\end{lemma}

\begin{proof}
	By the construction of Definition \ref{distlaw_comp} and Lemma \ref{distr_tringle_2}, 
\[
		\begin{tikzpicture}[node distance=2cm,auto]
			\node (A) {$R_jC$};
			\node (B)[right of=A]{$R_jR_iC$};
			\node (C)[below of=A]{$R_iR_jC$};
			\draw[->,swap] (A) to node {$\eta^i_{R_jC}$} (C);
			\draw[->](A) to node {$R_j\eta^i_C$}(B);
			\draw[->,swap] (C) to node {$\lambda^{i,j}_C$}(B);
		\end{tikzpicture}	
\]
	and 
	\[
	\begin{tikzpicture}[node distance=2cm]
		\node (B){$R_iC$};
		\node (C)[below of=B]{};
		\node (D)[left of=C]{$R_iR_jC$};
		\node (E)[right of=C]{$R_jR_iC$};
		\draw[->](B) to node[left=3]{$R_i\eta^j_C$}(D);
		\draw[->](B) to node[right=3]{$\eta^j_{R_iC}$}(E);
		\draw[->](D) to node[below=3]{$\lambda^{i,j}_C$}(E);
	\end{tikzpicture}
	\]
	both commute for every theory $C$. 
	
	\noindent We now diagram chase and obtain the equality
	\[
	\lambda^{i,j}\circ \mu^i_{R_j}\circ \eta^i_{R_iR_j}\circ\eta^i_{R_j}=\lambda^{i,j}\circ\eta^i_{R_j}=R_j\eta^i=R_j\mu^i\circ R_j\eta^i_{R_i}\circ R_j\eta^i
	\]
	\[
	=R_j\mu^i\circ\lambda^{i,j}_{R_i}\circ \eta^i_{R_jR_i}\circ R_j\eta^i=R_j\mu^i\circ\lambda^{i,j}_{R_i}\circ R_iR_j\eta^i\circ\eta^i_{R_j}
	\]
	\[
	=R_j\mu^i\circ\lambda^{i,j}_{R_i}\circ R_i\lambda^{i,j}\circ R_i\eta^i_{R_j}\circ\eta^i_{R_j}=R_j\mu^i\circ\lambda^{i,j}_{R_i}\circ R_i\lambda^{i,j}\circ \eta^i_{R_iR_j}\circ\eta^i_{R_j}
	\]
	of natural transformations. Let $C$ be a theory and let 
	\[
		\begin{tikzpicture}[node distance=2cm]
		\node (A) {$\vec{p}$};
			\node (B)[right of=A]{$i$};
	\draw[transform canvas={yshift=0.3ex},->] (A) to node[above=3] {$f$} (B);
			\draw[transform canvas={yshift=-0.3ex},->,swap](A) to node[below=3] {$g$} (B);
	\end{tikzpicture}
	\]
	be an admissible pair in $R_j(C)$. We are forced to have that
	\[
	(\lambda^{i,j}\circ \mu^i_{R_j}\circ \eta^i_{R_iR_j})_C(\delta_{f,g})=(R_j\mu^i\circ\lambda^{i,j}_{R_i}\circ R_i\lambda^{i,j}\circ \eta^i_{R_iR_j})_C(\delta_{f,g}).
	\]
	The universal property of Lemma \ref{univ_prop_monad_unit} forces
	\[
	(\lambda^{i,j}\circ \mu^i_{R_j}\circ \eta^i_{R_iR_j})_C=(R_j\mu^i\circ\lambda^{i,j}_{R_i}\circ R_i\lambda^{i,j}\circ \eta^i_{R_iR_j})_C.
	\]
	Similarly, the diagram
	\[
	\begin{tikzpicture}[node distance=2cm]
		\node (E)[below of=D]{$R_i^2R_j$};
		\node (F)[right of=E]{$R_iR_j$};
		\node (G)[right of=F]{$R_jR_i$};
		\node (H)[below of=E]{$R_iR_jR_i$};
		\node (I)[node distance=4cm,right of=H,swap]{$R_jR_i^2$};
		
		\draw[->](E) to node[above=3]{$\mu^i_{R_j}$}(F);
		\draw[->](F) to node[above=3]{$\lambda^{i,j}$}(G);
		\draw[->,swap] (E) to node[left=3]{$R_i\lambda^{i,j}$}(H);
		\draw[->,swap] (H) to node[below=3]{$\lambda^{i,j}_{R_i}$}(I);
		\draw[->](I) to node[right =3]{$R_j\mu^i$}(G);
	\end{tikzpicture}
	\]
commutes by the universal property of Lemma \ref{univ_prop_monad_unit}.  By a symmetric argument, the rectangle 
	\[
	\begin{tikzpicture}[node distance=2cm]
		\node (A) {$R_iR_j^2$};
		\node (B)[below of =A]{$R_iR_j$};
		\node (F) [right of=A]{$R_jR_iR_j$};
		\node (G) [right of=F]{$R_j^2R_i$};
		\node (H) [below of=G]{$R_jR_i$};
		\draw[->,swap] (A) to node[black,left=3] {$R_i\mu^j$}(B);
		\draw[->] (A) to node [black,above=3] {$\lambda^{i,j}_{R_j}$}(F);
		\draw[->] (F) to node[black,above=3] {$R_j\lambda^{i,j}$}(G);
		\draw[->](G) to node[black,right=3] {$\mu^jR_i$}(H);
		\draw[->](B) to node[black,below=3] {$\lambda^{i,j}$}(H);
	\end{tikzpicture}
	\]
	commutes. We have thus proven that $\lambda^{i,j}:R_iR_j\Rightarrow R_jR_i$ is a distributive law.
\end{proof}

\begin{lemma}\label{yang_baxt_gpds}
	The Yang-Baxter equation 
	
	\[
	\begin{tikzpicture}[node distance=1.5cm]
		\node (A){$R_iR_jR_k$};
		\node (B)[right of=A]{};
		\node (C)[above of =B]{$R_jR_iR_k$};
		\node (D)[below of =B]{$R_iR_kR_j$};
		\node (E)[right of=C,node distance=3cm]{$R_jR_kR_i$};
		\node (F)[right of=D,node distance=3cm]{$R_kR_iR_j$};
		\node (G)[right of=B]{};
		\node (H)[right of =G,node distance=3cm]{$R_kR_jR_i$};
		\draw[->](A) to node[left=3]{$\lambda^{i,j}R_k$}(C);
		\draw[->,swap](A) to node[left=3]{$R_i\lambda^{j,k}$}(D);
		\draw[->](C)to node[above=3]{$R_j\lambda^{i,k}$}(E);
		\draw[->,swap](D) to node[below=3]{$\lambda^{i,k}R_j$}(F);
		\draw[->](E) to node[right=3]{$\lambda^{j,k}R_i$}(H);
		\draw[->,swap](F) to node[right=3]{$R_k\lambda^{i,j}$}(H);
	\end{tikzpicture}
	\]
	
	is satisfied for $0\leq i<j<k$.
\end{lemma}

\begin{proof}
	Following through a diagram chase, we obtain that 
	\[
	\lambda^{j,k}_{R_i}\circ R_j\lambda^{i,k}\circ \lambda^{i,j}_{R_k}\circ \eta^i_{R_jR_k}=\lambda^{j,k}_{R_i}\circ R_j\lambda^{i,k}\circ R_j\eta^i_{R_k}
	\]
	\[
	=\lambda^{j,k}_{R_i}\circ R_jR_k\eta^i=R_kR_j\eta^i\circ \lambda^{j,k}
	\]
	\[
	=R_k\lambda^{i,j}\circ R_k\eta^i_{R_j} \circ \lambda^{j,k}=R_k\lambda^{i,j}\circ \lambda^{i,k}_{R_j}\circ \eta^i_{R_kR_j} \circ \lambda^{j,k}
	\]
	\[
	=R_k\lambda^{i,j}\circ \lambda^{i,k}_{R_j}\circ R_i\lambda^{j,k} \circ \eta^i_{R_jR_k}.
	\]
	Let $C$ be a theory and \[
		\begin{tikzpicture}[node distance=2cm]
		\node (A) {$\vec{p}$};
			\node (B)[right of=A]{$i$};
	\draw[transform canvas={yshift=0.3ex},->] (A) to node[above=3] {$f$} (B);
			\draw[transform canvas={yshift=-0.3ex},->,swap](A) to node[below=3] {$g$} (B);
	\end{tikzpicture}
	\]
	be an admissible pair in $R_j(R_k(C))$. We are forced to have that 
\[
(\lambda^{j,k}_{R_i}\circ R_j\lambda^{i,k}\circ \lambda^{i,j}_{R_k})_C(\delta_{f,g})=(R_k\lambda^{i,j}\circ \lambda^{i,k}_{R_j}\circ R_i\lambda^{j,k})_C(\delta_{f,g}).
\]
 Upon application of the universal property of Lemma \ref{univ_prop_monad_unit}, the diagram
	\[
	\begin{tikzpicture}[node distance=1.5cm]
		\node (A){$R_iR_jR_k$};
		\node (B)[right of=A]{};
		\node (C)[above of =B]{$R_jR_iR_k$};
		\node (D)[below of =B]{$R_iR_kR_j$};
		\node (E)[right of=C,node distance=3cm]{$R_jR_kR_i$};
		\node (F)[right of=D,node distance=3cm]{$R_kR_iR_j$};
		\node (G)[right of=B]{};
		\node (H)[right of =G,node distance=3cm]{$R_kR_jR_i$};
		\draw[->](A) to node[left=3]{$\lambda^{i,j}R_k$}(C);
		\draw[->,swap](A) to node[left=3]{$R_i\lambda^{j,k}$}(D);
		\draw[->](C)to node[above=3]{$R_j\lambda^{i,k}$}(E);
		\draw[->,swap](D) to node[below=3]{$\lambda^{i,k}R_j$}(F);
		\draw[->](E) to node[right=3]{$\lambda^{j,k}R_i$}(H);
		\draw[->,swap](F) to node[right=3]{$R_k\lambda^{i,j}$}(H);
	\end{tikzpicture}
	\] 
	must commute.
\end{proof}

\begin{theorem}\label{distr_ser_AWFS}
	The fibrant replacement monads of Notation \ref{fibr_repl_AWFS} may be upgraded to form a distributive series
	\[
	R:=((R_{k})_{k=0}^\infty,(\lambda^{i,j})_{j>i\geq 0})
	\]
\end{theorem}
\begin{proof} 
	Put the last three lemmas together.
\end{proof}

Let $(\hat{R},\hat{\eta})$ be the associated pointed endofunctor of the distributive series of Theorem \ref{distr_ser_AWFS} (see Definition \ref{ind_funct_on_dist}), which exists since $\Th_{\Theta_0^\op}$ is cocomplete. We now show the distributive series is completable. 

\begin{theorem}\label{ind_monad_AWFS}
	The structure $(\hat{R},\hat{\eta})$ extends to a monad structure on $\Th_{\Theta_0^\op}$. Moreover, the distributive series is completable. 
\end{theorem}
\begin{proof}
	We shall assign a multiplication $\hat{\mu}:\hat{R}^2\Rightarrow\hat{R}$. Define $\hat{\mu}_0:\hat{R}\to\hat{R}$ to be $\hat{\mu}_0=\id_{\hat{R}}$. Suppose $C$ is a theory and that 
	\[
		\begin{tikzpicture}[node distance=2cm]
		\node (A) {$\vec{p}$};
			\node (B)[right of=A]{$0$};
	\draw[transform canvas={yshift=0.3ex},->] (A) to node[above=3] {$f$} (B);
			\draw[transform canvas={yshift=-0.3ex},->,swap](A) to node[below=3] {$g$} (B);
	\end{tikzpicture}
	\]
	is an admissible pair in $\hat{R}(C)$. There is a unique map $\hat{\mu}_1:R_0\hat{R}\to\hat{R}$ where we send the choice of lift to the admissible pair above in $R_0(\hat{R}(C))$ added along $\eta^0_{\hat{R}C}$ to the choice of lift in $\hat{R}(C)$ and such that the diagram
	\[
	\begin{tikzpicture}[node distance=3cm]
		\node (A) {$\hat{R}(C)$};
		\node (B)[below of=A]{$R_0(\hat{R}(C))$};
		\node (E) [right of=A,node distance=2cm]{$\hat{R}(C)$};
		\draw[->,swap](A) to node[black,left=3] {$\eta^0_{\hat{R}(C)}$} (B);
		\draw[->] (A) to node[black, above=3]{$\id_{\hat{R}(C)}$} (E);
		\draw[->,dashed] (B) to node[right=5] {$(\hat{\mu}_1)_C$} (E);
	\end{tikzpicture}
	\]
commutes by the universal property of Lemma \ref{univ_prop_monad_unit}. Moreover, $\hat{\mu}_1$ assembles into a natural transformation. We repeat what we just did inductively to obtain a diagram
	\[
	\begin{tikzpicture}[node distance=2cm]
		\node (A) {$\hat{R}$};
		\node (B)[below of=A]{$R_0\hat{R}$};
		\node (C)[below of=B]{$R_1R_0\hat{R}$};
		\node (D) [below of=C]{$\vdots$};
		\node (E) [right of=A,node distance=2cm]{$\hat{R}$};
		\draw[->,swap](A) to node[black,left=3] {$\eta^0_{\hat{R}}$} (B);
		\draw[->,swap](B) to node[black,left=3] {$\eta^1_{R_0\hat{R}}$} (C);
		\draw[->,swap](C) to node[black, left=3] {$\eta^2_{R_1R_0\hat{R}}$} (D);
		\draw[->] (A) to node[black, above=3]{$1_{\hat{R}}$} (E);
		\draw[->,dashed] (B) to node[right=3] {$\hat{\mu}_1$} (E);
		\draw[->,dashed,bend right=30] (C) to node[right=3] {$\hat{\mu}_2$} (E);
	\end{tikzpicture}
	\]
	of natural transformations. Upon taking colimits, we induce a unique natural transformation $\hat{\mu}:\hat{R}^2\Rightarrow\hat{R}$ such that the diagram
	\[
	\begin{tikzpicture}[node distance=2cm]
		\node (A) {$\hat{R}$};
		\node (B)[below of=A]{$\hat{R}^2$};
		\node (E) [right of=A,node distance=2cm]{$\hat{R}$};
		\draw[->,swap](A) to node[black,left=3] {$\hat{\eta}_{\hat{R}}$} (B);
		\draw[->] (A) to node[black, above=3]{$1_{\hat{R}}$} (E);
		\draw[->,dashed] (B) to node[right=5] {$\hat{\mu}$} (E);
	\end{tikzpicture}
	\]
	commutes. Notice that
	\[
	\hat{\mu}\circ \hat{\eta}_{\hat{R}}\circ \hat{\eta}=\hat{\mu}\circ \hat{R}\hat{\eta}\circ \hat{\eta}.
	\]
	By construction and inductive use of the universal property of Lemma \ref{univ_prop_monad_unit}, we must have that 
	\[
	\hat{\mu}\circ \hat{R}\hat{\eta}=\hat{\mu}\circ \hat{\eta}_{\hat{R}}=1_{\hat{R}}.
	\]
Similarly, the equations
	\[
	\hat{\mu}\circ \hat{R}\hat{\mu}\circ \hat{\eta}_{\hat{R}^2}=\hat{\mu}\circ \hat{\eta}_{\hat{R}}\circ \hat{\mu}
	\]
	\[
	=\hat{\mu}=\hat{\mu}\circ \hat{\mu}_{\hat{R}}\circ \hat{\eta}_{\hat{R}^2}
	\]
are forced to hold for the same reason.  By construction and inductive use of the universal property (see Lemma \ref{univ_prop_monad_unit}), we have that 
	\[
	\hat{\mu}\circ \hat{R}\hat{\mu}=
	\hat{\mu}\circ \hat{\mu}_{\hat{R}}.
	\]
	Therefore $(\hat{R},\hat{\eta},\hat{\mu})$ is a monad. 

Let $n\geq 0$. Recall the $n$-tail $\kappa_n$ of \ref{n-tail}. Notice that
\[
\hat{\mu}\circ \kappa_n\kappa_n\circ \hat{\eta}^n_{\hat{R}_n}=\hat{\mu}\circ \hat{R}\kappa_n\circ (\kappa_n)_{\hat{R}_n}\circ\hat{\eta}^n_{\hat{R}_n}
\]
\[
=\hat{\mu}\circ (\kappa_n)_{\hat{R}}\circ \hat{R}_n\kappa_n\circ\hat{\eta}^n_{\hat{R}_n}=\hat{\mu}\circ (\kappa_n)_{\hat{R}}\circ \hat{\eta}^n_{\hat{R}}\circ\kappa_n
\]
\[
=\hat{\mu}\circ\hat{\eta}_{\hat{R}}\circ\kappa_n=\kappa_n=\kappa_n\circ\hat{\mu}^n\circ \hat{\eta}^n_{\hat{R}_n}
\]

By inductive use of the universal property (see Lemma \ref{univ_prop_monad_unit}), we have that the diagram
	\[
\begin{tikzpicture}[node distance=2cm,auto]
	\node (A) {$\hat{R}_n^2$};
	\node (B)[right of=A] {$\hat{R}^2$};
	\node(C)[below of=A] {$\hat{R}_n$};
	\node (D)[right of=C] {$\hat{R}$};
	\draw[->] (A) to node {$\kappa_n\kappa_n$}(B);
	\draw[->,swap] (A) to node{$\hat{\mu}^n$}(C);
	\draw[->] (B) to node{$\hat{\mu}$}(D);
	\draw[->,swap] (C) to node {$\kappa_n$}(D);
\end{tikzpicture}
\]
commutes. Therefore the distributive series is completable as required.
\end{proof}

\begin{remark}
We made mention of the inductive use of the universal property. We can do this because $\hat{\eta}$ is the colimit of the diagram
\[
1_{\Th_{\Theta_0^\op}}\xrightarrow{\eta^0} R_0\xrightarrow{\eta^1_{R_0}} R_1R_0\xrightarrow{\eta^2_{R_1R_0}}\cdots
\]
in the category of endofunctors on $\Th_{\Theta_0^\op}$ and because we argue dimension-by-dimension to show that both ways of sending the appropriate choice of lift added around the required diagram are in fact the same. 
\end{remark}
 
\begin{definition}\label{our_sp_choice}
	The \emph{inductive $(\infty,0)$-coherator} 
	\[
	J^{IC}:\Theta_0^\op\to IC
	\]
	is defined to be
	\[
	J^{SC}:=\widehat{R}(\id_{\Theta_0^\op}).
	\]
\end{definition}

\begin{theorem} \label{our_sp_choice_proof}
The theory $J^{IC}:\Theta_0^\op\to IC$ of Definition \ref{our_sp_choice} is an $(\infty,0)$-coherator.
\end{theorem}

\begin{proof}
This is an $(\infty,0)$-coherator by construction.
\end{proof}

\begin{remark}
The construction of this coherator is suggested to exist informally in the paragraph following Definition 2.5 of \cite{HenLan}.
\end{remark}

\section{Strict Infinity Groupoids}\label{th_for_st_inf_grpds_const}
We finish the main exposition of this paper by showing that a theory for strict $\infty$-groupoids can be obtained in the same way we obtained the inductive coherator. The only difference being that we apply Kelly's Small Object Argument to obtain strong factorization systems. We reference proofs there for proofs here to avoid repetitiveness. 
\begin{definition}
We say that a globular theory $J:\Theta_0^\op\to C$ is \emph{uniquely contractible} if every admissible pair has a unique lift.
\end{definition}
\begin{remark}
Ara calls this canonically contractible in 2.9 of \cite{Dim4}.
\end{remark}
\begin{notation} \label{fibr_repl_OFS}
Let $(S_k,\eta^k,\mu^k)$ be the fibrant replacement monad corresponding to the strong factorization system obtained using Kelly's Small Object Argument (see Theorem \ref{Kelly_SOA}) to the set of maps $I_k$ for all $k\geq 0$.
\end{notation}

\begin{lemma}\label{univ_prop_monad_unit_st}
	Given a theory $C$, a map $F:C\to D$ of globular theories, and given a unique choice of lift $\delta_{Ff,Fg}:\vec{p}\to k+1$ in $D$ for the image under $F$ of every admissible pair of the form
\[
		\begin{tikzpicture}[node distance=2cm]
			\node (A) {$\vec{p}$};
			\node (B)[right of=A]{$k$};
			\draw[transform canvas={yshift=0.3ex},->] (A) to node[above=3] {$f$} (B);
			\draw[transform canvas={yshift=-0.3ex},->,swap](A) to node[below=3] {$g$} (B);
		\end{tikzpicture}
		\]
	in $C$, there is a unique map
	\[
	F':S_kC\to D
	\]
	such that
	\[
	F'\circ\eta^k_C=F
	\]
	and
	\[
	F'(\delta_{f,g})=\delta_{Ff,Fg}
	\]
	for every admissible pair of the form
\[
		\begin{tikzpicture}[node distance=2cm]
			\node (A) {$\vec{p}$};
			\node (B)[right of=A]{$k$};
			\draw[transform canvas={yshift=0.3ex},->] (A) to node[above=3] {$f$} (B);
			\draw[transform canvas={yshift=-0.3ex},->,swap](A) to node[below=3] {$g$} (B);
		\end{tikzpicture}
		\]
	in $C$.
\end{lemma}
\begin{proof}
This is a reinterpretation of Lemma \ref{Uni_prop_of_Kell's_soa} in this setting.
\end{proof}

\begin{definition} \label{distlaw_comp_2}
	Given a globular theory $C$, we define a map
	\[
	\kappa^{i,j}_C: S_iS_jC\Rightarrow S_jS_iC
	\]
	
	for $0\leq i<j$ as follows:
	\begin{itemize}
		\item we consider the fork 
		\[
		\begin{tikzpicture}[node distance=2cm,auto]
			\node (A) {$S_jC$};
			\node (B)[right of=A]{$S_jS_iC$};
			\node (C)[below of=A]{$S_iS_jC$};
			\draw[->,swap] (A) to node {$\eta^i_{S_jC}$} (C);
			\draw[->](A) to node {$S_j\eta^i_C$}(B);
		\end{tikzpicture}
		\]
		\item every admissible pair of the form 
		\[
		\begin{tikzpicture}[node distance=2cm]
			\node (A) {$\vec{p}$};
			\node (B)[right of=A]{$i$};
			\draw[transform canvas={yshift=0.3ex},->] (A) to node[above=3] {$f$} (B);
			\draw[transform canvas={yshift=-0.3ex},->,swap](A) to node[below=3] {$g$} (B);
		\end{tikzpicture}
		\]
		in $R_jC$ induces an admissible pair with the same name in $R_jR_iC$ and has a unique choice of lift in $S_jS_iC$
		
		\item let $\kappa^{i,j}_C:S_iS_jC\to S_jS_iC$ be the unique map such that the diagram
		\begin{equation}
			\begin{tikzpicture}[node distance=2cm,auto]
				\node (A) {$S_jC$};
				\node (B)[right of=A]{$S_jS_iC$};
				\node (C)[below of=A]{$S_iS_jC$};
				\draw[->,swap] (A) to node {$\eta^i_{S_jC}$} (C);
				\draw[->](A) to node {$S_j\eta^i_C$}(B);
				\draw[->,swap] (C) to node {$\kappa^{i,j}_C$}(B);
			\end{tikzpicture}
		\end{equation}
		commutes by sending the unique choice of lift added along $\eta^i_{S_jC}$ of the admissible pair 
\[
		\begin{tikzpicture}[node distance=2cm]
			\node (A) {$\vec{p}$};
			\node (B)[right of=A]{$i$};
			\draw[transform canvas={yshift=0.3ex},->] (A) to node[above=3] {$f$} (B);
			\draw[transform canvas={yshift=-0.3ex},->,swap](A) to node[below=3] {$g$} (B);
		\end{tikzpicture}
		\]
		to the unique choice of lift in $S_jS_iC$.
	\end{itemize}
\end{definition}

\begin{lemma}\label{distr_tringle_3}
	For $0\leq i<j$ and every theory $C$, the triangle 
	\[
	\begin{tikzpicture}[node distance=2cm]
		\node (B){$S_iC$};
		\node (C)[below of=B]{};
		\node (D)[left of=C]{$S_iS_jC$};
		\node (E)[right of=C]{$S_jS_iC$};
		\draw[->](B) to node[left=3]{$S_i\eta^j_C$}(D);
		\draw[->](B) to node[right=3]{$\eta^j_{S_iC}$}(E);
		\draw[->](D) to node[below=3]{$\kappa^{i,j}_C$}(E);
	\end{tikzpicture}
	\]
	commutes.
\end{lemma}
\begin{proof}
Repeat argument from Lemma \ref{distr_tringle_2}.
\end{proof}

\begin{lemma}\label{dist_law_ij_2}
	The maps $\kappa^{i,j}_C:S_iS_jC\to S_jS_iC$ of \ref{distlaw_comp_2} form a natural transformation
	\[
	\kappa^{i,j}: S_iS_j\Rightarrow S_jS_i
	\]
	for $0\leq i<j$.
\end{lemma}

\begin{proof}
Repeat argument from Lemma \ref{dist_law_ij}.
\end{proof}

\begin{lemma}\label{distr_ser_fibr_repl_2}
	For $0\leq i<j$, $\kappa^{i,j}:S_iS_j\Rightarrow S_jS_i$ is a distributive law.	
\end{lemma}

\begin{proof}
Repeat argument from Lemma \ref{distr_ser_fibr_repl}.
\end{proof}

\begin{lemma}\label{yang_baxt_gpds_2}
	The Yang-Baxter equation 
	
	\[
	\begin{tikzpicture}[node distance=1.5cm]
		\node (A){$S_iS_jS_k$};
		\node (B)[right of=A]{};
		\node (C)[above of =B]{$S_jS_iS_k$};
		\node (D)[below of =B]{$S_iS_kS_j$};
		\node (E)[right of=C,node distance=3cm]{$S_jS_kS_i$};
		\node (F)[right of=D,node distance=3cm]{$S_kS_iS_j$};
		\node (G)[right of=B]{};
		\node (H)[right of =G,node distance=3cm]{$S_kS_jS_i$};
		\draw[->](A) to node[left=3]{$\kappa^{i,j}S_k$}(C);
		\draw[->,swap](A) to node[left=3]{$S_i\kappa^{j,k}$}(D);
		\draw[->](C)to node[above=3]{$S_j\kappa^{i,k}$}(E);
		\draw[->,swap](D) to node[below=3]{$\kappa^{i,k}S_j$}(F);
		\draw[->](E) to node[right=3]{$\kappa^{j,k}S_i$}(H);
		\draw[->,swap](F) to node[right=3]{$S_k\kappa^{i,j}$}(H);
	\end{tikzpicture}
	\]
	
	is satisfied for $0\leq i<j<k$.
\end{lemma}

\begin{proof}
Repeat argument from Lemma \ref{yang_baxt_gpds}.
\end{proof}

\begin{theorem}\label{distr_ser_OFS}
	The fibrant replacement monads of Notation \ref{fibr_repl_OFS} may be upgraded to form a distributive series
	\[
	S:=((S_{k})_{k=0}^\infty,(\kappa^{i,j})_{j>i\geq 0})
	\]
\end{theorem}
\begin{proof} 
	Put the last three lemmas together.
\end{proof}

Let $(\hat{S},\hat{\eta})$ be the associated pointed endofunctor of the distributive series of Theorem \ref{distr_ser_OFS} (see Definition \ref{ind_funct_on_dist}). 

\begin{lemma}\label{ind_monad_OFS}
	The structure $(\hat{S},\hat{\eta})$ extends to a monad structure on $\Th_{\Theta_0^\op}$. Moreover, the distributive series is completable. 
\end{lemma}

\begin{proof}
Repeat argument from Theorem \ref{ind_monad_AWFS}.
\end{proof}

\begin{definition}
	We define the \emph{strict coherator}
	\[
	J^{\st}:\Theta_0^\op\to \hat{\Theta}
	\]
	to be
	\[
	J^{\st}:=\widehat{S}(\id_{\Theta_0^\op}).
	\]
\end{definition}

\begin{theorem}\label{strict_coh_is_true}
The category of strict $\infty$-groupoids is equivalent to 
\[
\Mod(\hat{\Theta}).
\]
\end{theorem}

\begin{proof}
The strict coherator
	\[
	J^{\st}:\Theta_0^\op\to \hat{\Theta}
	\]
satisfies the universal property of Proposition 3.8 of \cite{Dim5}. By Proposition 3.22 of \cite{Dim5}, $\Mod(\hat{\Theta})$ is equivalent to the category of strict $\infty$-groupoids.
\end{proof}

\section{Appendix A: Understanding the Inductive Coherator}
We use this appendix to show why moving along the coherator corresponds to weak enrichment of the shapes for higher groupoids. A similar find can be found in \cite{Dim1} but organized in a different fashion due to the coherator being different and us being interested how this relates to weak enrichment. We choose to write this part to provide a place where the interested reader can go back and look at what certain notation means when dicussing maps in our special theory.
\begin{notation}
	Let $J^{IC}:\Theta_0^\op\to IC$ be our special theory. We use the following notation:
	\begin{itemize}
		\item $\circ$ is always used to denote composition on $SC$
		\item $s$ and $t$ refers to the application of the source and target maps once and $s^n$ and $t^n$ refers to an $n$-fold composition of source and target maps, respectively.
		\item given $\vec{p}=(p_1,\dots,p_{2k+1})$, we use 
		\[
		\epsilon_i:\vec{p}\to p_{2i+1}
		\]
		to denote the projection for $i=0,1,\dots, n$
		\item we use $[,\dots,]$ to denote that we are forming the unique map formed by the pullback. Just to make thinks clear, here is a toy example. Suppose $f:4\to 2$ and $g:4\to 1$ are maps of $IC$ such that $s^2\circ f=t\circ g$. Then there is a unique map 
		\[
		[g,f]:4\to (1,0,2),
		\]
	\end{itemize}
\end{notation}
\noindent This will help the reader keep track of some of the crucial notation. We will introduce more as needed. We officially begin our investigation here, let $X:IC\to\Set$ be an $\infty$-groupoid of shape $IC$. We now study what types of maps are added along the extension
\[
J_1:\Theta_0^\op\to IC_1.
\]

\underline{\textbf{Identity on $0$-cells}}
The map $Z:0\to 1$ which  is obtained as a lift of the admissible pair
\[
		\begin{tikzpicture}[node distance=2.5cm]
			\node (A) {$0$};
			\node (B)[right of=A]{$0$};
			\draw[transform canvas={yshift=0.3ex},->] (A) to node[above=3] {$1_0$} (B);
			\draw[transform canvas={yshift=-0.3ex},->,swap](A) to node[below=3] {$1_0$} (B);
		\end{tikzpicture}
		\]
admits a $1$-cell
\[
x\xrightarrow{1_x}x
\]
for all $0$-cells $x\in X_0$.

\underline{\textbf{Composition of $1$-cells along $0$-cells}}
The map $c:(1,0,1)\to 1$  is obtained as a lift of the admissible pair
\[
		\begin{tikzpicture}[node distance=2.5cm]
			\node (A) {$(1,0,1)$};
			\node (B)[right of=A]{$0$};
			\draw[transform canvas={yshift=0.3ex},->] (A) to node[above=3] {$s\circ \epsilon_1$} (B);
			\draw[transform canvas={yshift=-0.3ex},->,swap](A) to node[below=3] {$t\circ \epsilon_0$} (B);
		\end{tikzpicture}
		\]
and admits a $1$-cell
\[
x\xrightarrow{g\circ f}z
\]
for all $(g,f)\in X_1\times_0X_1$ where $tg=z$ and $sf=x$.

\underline{\textbf{Inverse of $1$-cells along $0$-cells}}
The map $\omega:1\to 1$ is obtained as a lift of the admissible pair
\[
		\begin{tikzpicture}[node distance=2.5cm]
			\node (A) {$1$};
			\node (B)[right of=A]{$0$};
			\draw[transform canvas={yshift=0.3ex},->] (A) to node[above=3] {$t$} (B);
			\draw[transform canvas={yshift=-0.3ex},->,swap](A) to node[below=3] {$s$} (B);
		\end{tikzpicture}
		\]
and admits a $1$-cell
\[
y\xrightarrow{f^{-1}}x
\]
for all $f\in X_1$ where $tf=y$ and $sf=x$.

\noindent Now if $X$ is $1$-truncated (see Definition 2.2 of \cite{HenLan}), then these are all the generating structure maps. Moreover, this data satisfies the axioms of a groupoid.  We now study what types of maps are added along the extension
\[
J_2:IC_1\to IC_2
\]
and say what it means for $X$ to be $2$-truncated.

\underline{Associativity of Composition of $1$-cells along $0$-cells}
The map $a:(1,0,1,0,1)\to 2$ is obtained as a lift of the admissible pair
\[
		\begin{tikzpicture}[node distance=2.5cm]
			\node (A) {$(1,0,1,0,1)$};
			\node (B)[right of=A]{$1$};
			\draw[transform canvas={yshift=0.3ex},->] (A) to node[above=3] {$c\circ(c{,}1)$} (B);
			\draw[transform canvas={yshift=-0.3ex},->,swap](A) to node[below=3] {$c\circ (1{,}c)$} (B);
		\end{tikzpicture}
		\]
and admits a $2$-cell 
\[
(h\circ g)\circ f\xrightarrow{a}h\circ (g\circ f)
\]
for all $(h,g,f)\in X_1\times_0X_1\times_0X_1$.

\underline{\textbf{Left Unitor of Composition of $1$-cells along $0$-cells}}
The map $Z_l:1\to 2$ is obtained as a lift of the admissible pair
\[
		\begin{tikzpicture}[node distance=2.5cm]
			\node (A) {$1$};
			\node (B)[right of=A]{$1$};
			\draw[transform canvas={yshift=0.3ex},->] (A) to node[above=3] {$c\circ(Z\circ t{,}1)$} (B);
			\draw[transform canvas={yshift=-0.3ex},->,swap](A) to node[below=3] {$1_1$} (B);
		\end{tikzpicture}
		\]
and admits a $2$-cell 
\[
1_y\circ f\xrightarrow{l}f
\]
for all $f\in X_1$ where $tf=y$.

\underline{\textbf{Right Unitor of Composition of $1$-cells along $0$-cells}}
The map $Z_r:1\to 2$  is obtained as a lift of the admissible pair
\[
		\begin{tikzpicture}[node distance=2.5cm]
			\node (A) {$1$};
			\node (B)[right of=A]{$1$};
			\draw[transform canvas={yshift=0.3ex},->] (A) to node[above=3] {$c\circ(1,Z\circ s)$} (B);
			\draw[transform canvas={yshift=-0.3ex},->,swap](A) to node[below=3] {$1_1$} (B);
		\end{tikzpicture}
		\]
and admits a $2$-cell 
\[
f\circ 1_x\xrightarrow{l}f
\]
for all $f\in X_1$ where $sf=x$.

We may keep going to obtains lifts of admissible pairs which admits new $1$-cells which witness an unbiased composition and $2$-cells for:
\begin{itemize}
	\item Witness of Left Invertibility of $1$-cells along $0$-cells
	\item Witness of Right Invertibility of $1$-cells along $0$-cells
	\item Composition of $2$-cells along $1$-cells
	\item Composition of $2$-cells along $0$-cells
	\item Identity of Identity on $0$-cells
	\item Identity on $1$-cells
	\item Inverse of $2$-cells along $0$-cells
	\item Inverse of $2$-cells along $1$-cells
	\item interchange of $2$-cells
	\item etc...
\end{itemize}
There are many maps obtained by adding lifts along the extension
\[
J_2:IC_1\to IC_2.
\]
Upon examination, we say that if $X$ is $2$-truncated (see Definition 2.2 of \cite{HenLan}) then $X$ has the structure of an unbiased weak $2$-groupoid. More generally, when we move one position along the coherator, we weakly enrich one level categorically and the previous level has maps which adds an unbiased structure on that level. 

\section{Appendix B: Epi-Monads}
\begin{definition}\label{epi_monads}
We say that a monad $(T,\eta,\mu)$ on a category $C$ is an \emph{epi-monad} if $\eta_x$ is an epimorphism in $C$ for all $x\in\ob(C)$.
\end{definition}

\begin{lemma}\label{epi_monad_is_epi_in _end}
If $(T,\eta,\mu)$ is an epi-monad on a category $C$, then $\eta$ is an epimorphism in the category of endofunctors on $C$.
\end{lemma}

\begin{proof}
Let $\alpha,\alpha':T\Rightarrow F$ be maps in $\End(C)$ such that 
\[
\alpha\circ \eta=\alpha'\circ \eta.
\]
Then
\[
\alpha_c\circ \eta_c=\alpha'_c\circ\eta_c
\]
for all objects $c$ of $C$ and therefore $\alpha_c=\alpha'_c$ for all objects $c$ of $C$. This means that $\alpha=\alpha'$ and $\eta$ is an epimorphism in $\End(C)$.
\end{proof}

\begin{definition}
Let 
\[
	T:=((T_i)_{i=0}^\infty,(\lambda_{i,j})_{i>j})
	\]
be a distributive series of monads on $C$. We call $T$ an \emph{epi-distributive series of monads} if $(T_i,\eta^i,\mu^i)$ is an epi-monad for all $i\geq 0$.
\end{definition}

We need the following lemma.

\begin{lemma}\label{comp_of_epis_is_epi}
Let $C$ be a category  and $\xi:\omega\to C$ be a diagram, where $\omega$ is the ordinal for the poset of natural numbers, whose colimit exists and 
\[
\xi(i)\to \xi(i+1)
\]
is an epimorphism for all $i\geq 0$. Then the induced map 
\[
\xi(0)\to \colim(\xi)
\]
is an epimorphism.
\end{lemma}

\begin{lemma}\label{ind_pointing_is_epi}
Let 
\[
	T:=((T_i)_{i=0}^\infty,(\lambda_{i,j})_{i>j})
	\]
be an epi-distributive series of monads on $C$. Then the pointing
\[
	\hat{\eta}:1_C\Rightarrow \hat{T}
	\]
for the associated pointed endofunctor of the distributive series is an epimorphism in $\End(C)$. More specifically, the pointing 
\[
	\hat{\eta}:1_C\Rightarrow \hat{T}
	\]
is an epimorphism on components.
\end{lemma}

\begin{proof}
This is a combination of Lemma \ref{epi_monad_is_epi_in _end} and Lemma \ref{comp_of_epis_is_epi}.
\end{proof}

\begin{theorem}
Let 
\[
	T:=((T_i)_{i=0}^\infty,(\lambda_{i,j})_{i>j})
	\]
be an epi-distributive series of monads on a category $C$, in which the induced pointed endofunctor exists. Then the following equivalent conditions hold:
\begin{itemize}
\item  the associated pointed endofunctor of the distributive series lifts to a monad on $C$
\item the associated pointed endofunctor of the distributive series lifts to a unique monad on $C$
\item the distributive series is completable.
\end{itemize}
\end{theorem}

\begin{proof}
Suppose $T$ is completable. Then the associated pointed endofunctor of the distributive series lifts to a monad on $C$.

Suppose  the associated pointed endofunctor of the distributive series lifts to a unique monad on $C$. Then the associated pointed endofunctor of the distributive series lifts to a monad on $C$.

Suppose the associated pointed endofunctor of the distributive series lifts to a monad on $C$. Suppose that $(\hat{T},\hat{\eta},\hat{\mu})$ and $(\hat{T},\hat{\eta},\ol{\mu})$
are two such lifts. Then notice that
\[
\hat{m}_c\circ \hat{\eta}_{\hat{T}(c)}=\ol{m}_c\circ \hat{\eta}_{\hat{T}(c)}
\]
for all objects $c\in\ob(C)$, so that $\hat{m}_c=\ol{m}_c$ for all objects $c$ of $C$ by Lemma \ref{ind_pointing_is_epi}. We conclude that $\hat{m}=\ol{m}$. This implies that  the associated pointed endofunctor of the distributive series lifts to a unique monad on $C$.

Suppose the associated pointed endofunctor of the distributive series lifts to a unique monad $(\hat{T},\hat{\eta},\hat{\mu})$ on $C$. Fix $n\geq 0$ and notice that the diagram
	\[
\begin{tikzpicture}[node distance=2.5cm,auto]
	\node (A) {$\hat{T}_n^2$};
	\node (B)[right of=A] {$\hat{T}_{n+1}^2$};
	\node(C)[below of=A] {$\hat{T}_n$};
	\node (D)[right of=C] {$\hat{T}_{n+1}$};
	\draw[->] (A) to node {$\eta^{n+1}_{\hat{T}_n}\eta^{n+1}_{\hat{T}_n}$}(B);
	\draw[->,swap] (A) to node{$\hat{\mu}^n$}(C);
	\draw[->] (B) to node{$\hat{\mu}^{n+1}$}(D);
	\draw[->,swap] (C) to node {$\eta^{n+1}_{\hat{T}_n}$}(D);
\end{tikzpicture}
\]
commutes since
\[
\eta^{n+1}_{\hat{T}_n}\circ \hat{\mu}^n=T_{n+1}\hat{\mu}^n\circ \eta^{n+1}_{\hat{T}^2_n}
\]
\[
= T_{n+1}\hat{\mu}^n\circ \mu^{n+1}_{\hat{T}^2}\circ T_{n+1}\eta^{n+1}_{\hat{T}_n^2}\circ \eta^{n+1}_{\hat{T}^2_n}
\]
\[
=T_{n+1}\hat{\mu}^n\circ \mu^{n+1}_{\hat{T}^2}\circ T_{n+1}\sigma_{\hat{T}_n}\circ T_{n+1}\hat{T}_n\eta^{n+1}\hat{T}_n
\circ \eta^{n+1}_{\hat{T}^2_n}
\]
\[
=\hat{\mu}^{n+1}\circ \hat{T}_{n+1}\eta^{n+1}\hat{T}_n\circ \eta^{n+1}_{\hat{T}^2_n}
\]
\[
=\hat{\mu}^{n+1}\circ \eta^{n+1}_{\hat{T}_n}\eta^{n+1}_{\hat{T}_n}
\]
where the first equality is naturality, the second one is by the monad axioms, the third one is by using the induced distributive law
\[
\sigma:\hat{T}_nT_{n+1}\Rightarrow T_{n+1}\hat{T}_n
\]
 from Theorem 2.1 of \cite{EC1}, the fourth one is just the definition of $\hat{\mu}^{n+1}$ and $\hat{T}_{n+1}$, and the fifth one is by the definition of $\eta^{n+1}_{\hat{T}_n}\eta^{n+1}_{\hat{T}_n}$.  By the universal property of the colimit and definition of $\kappa_n$, there is a unique natural transformation 
\[
\ol{\mu}:\hat{T}^2\Rightarrow\hat{T}
\]
such that 
\[
\begin{tikzpicture}[node distance=2cm,auto]
	\node (A) {$\hat{T}_n^2$};
	\node (B)[right of=A] {$\hat{T}^2$};
	\node(C)[below of=A] {$\hat{T}_n$};
	\node (D)[right of=C] {$\hat{T}$};
	\draw[->] (A) to node {$\kappa_n\kappa_n$}(B);
	\draw[->,swap] (A) to node{$\hat{\mu}^n$}(C);
	\draw[->] (B) to node{$\ol{\mu}$}(D);
	\draw[->,swap] (C) to node {$\kappa_n$}(D);
\end{tikzpicture}
\]
commutes for $n\geq 0$. In particular, we have that
\[
\begin{tikzpicture}[node distance=2cm,auto]
	\node (A) {$\hat{T}_0^2$};
	\node (B)[right of=A] {$\hat{T}^2$};
	\node(C)[below of=A] {$\hat{T}_0$};
	\node (D)[right of=C] {$\hat{T}$};
	\draw[->] (A) to node {$\kappa_0\kappa_0$}(B);
	\draw[->,swap] (A) to node{$\hat{\mu}^0$}(C);
	\draw[->] (B) to node{$\ol{\mu}$}(D);
	\draw[->,swap] (C) to node {$\kappa_0$}(D);
\end{tikzpicture}
\]
commutes, so that we must have that
\[
\begin{tikzpicture}[node distance=2cm,auto]
	\node (A) {$1_C$};
	\node (B)[right of=A] {$\hat{T}^2$};
	\node(C)[below of=A] {$1_C$};
	\node (D)[right of=C] {$\hat{T}$};
	\draw[->] (A) to node {$\hat{\eta}\hat{\eta}$}(B);
	\draw[->,swap] (A) to node{$1_{1_C}$}(C);
	\draw[->] (B) to node{$\ol{\mu}$}(D);
	\draw[->,swap] (C) to node {$\hat{\eta}$}(D);
\end{tikzpicture}
\]
commutes which says that
\[
\ol{\mu}\circ \hat{\eta}_{\hat{T}}\circ\hat{\eta}=\ol{\mu}\circ \hat{T}\hat{\eta}\circ\hat{\eta}
\]
\[
=\ol{\mu}\circ \hat{\eta}\hat{\eta}=\hat{\eta}.
\]
By Lemma \ref{ind_pointing_is_epi}, we are forced to have
\[
\ol{\mu}\circ \hat{\eta}_{\hat{T}}=1_{\hat{T}}=\hat{\mu}\circ \hat{\eta}_{\hat{T}}.
\]
By Lemma \ref{ind_pointing_is_epi}, we are forced to have $\ol{\mu}=\hat{\mu}$ which means that 
\[
\begin{tikzpicture}[node distance=2cm,auto]
	\node (A) {$\hat{T}_n^2$};
	\node (B)[right of=A] {$\hat{T}^2$};
	\node(C)[below of=A] {$\hat{T}_n$};
	\node (D)[right of=C] {$\hat{T}$};
	\draw[->] (A) to node {$\kappa_n\kappa_n$}(B);
	\draw[->,swap] (A) to node{$\hat{\mu}^n$}(C);
	\draw[->] (B) to node{$\hat{\mu}$}(D);
	\draw[->,swap] (C) to node {$\kappa_n$}(D);
\end{tikzpicture}
\]
commutes for $n\geq 0$, so that the distributive series is completable.

We have now shown the statements are equivalent. We now show that they hold.  Fix $n\geq 0$ and notice that the diagram
	\[
\begin{tikzpicture}[node distance=2.5cm,auto]
	\node (A) {$\hat{T}_n^2$};
	\node (B)[right of=A] {$\hat{T}_{n+1}^2$};
	\node(C)[below of=A] {$\hat{T}_n$};
	\node (D)[right of=C] {$\hat{T}_{n+1}$};
	\draw[->] (A) to node {$\eta^{n+1}_{\hat{T}_n}\eta^{n+1}_{\hat{T}_n}$}(B);
	\draw[->,swap] (A) to node{$\hat{\mu}^n$}(C);
	\draw[->] (B) to node{$\hat{\mu}^{n+1}$}(D);
	\draw[->,swap] (C) to node {$\eta^{n+1}_{\hat{T}_n}$}(D);
\end{tikzpicture}
\]
commutes since
\[
\eta^{n+1}_{\hat{T}_n}\circ \hat{\mu}^n=T_{n+1}\hat{\mu}^n\circ \eta^{n+1}_{\hat{T}^2_n}
\]
\[
= T_{n+1}\hat{\mu}^n\circ \mu^{n+1}_{\hat{T}^2}\circ T_{n+1}\eta^{n+1}_{\hat{T}_n^2}\circ \eta^{n+1}_{\hat{T}^2_n}
\]
\[
=T_{n+1}\hat{\mu}^n\circ \mu^{n+1}_{\hat{T}^2}\circ T_{n+1}\sigma_{\hat{T}_n}\circ T_{n+1}\hat{T}_n\eta^{n+1}\hat{T}_n
\circ \eta^{n+1}_{\hat{T}^2_n}
\]
\[
=\hat{\mu}^{n+1}\circ \hat{T}_{n+1}\eta^{n+1}\hat{T}_n\circ \eta^{n+1}_{\hat{T}^2_n}
\]
\[
=\hat{\mu}^{n+1}\circ \eta^{n+1}_{\hat{T}_n}\eta^{n+1}_{\hat{T}_n}
\]
where the first equality is naturality, the second one is by the monad axioms, the third one is by using the induced distributive law
\[
\sigma:\hat{T}_nT_{n+1}\Rightarrow T_{n+1}\hat{T}_n
\]
 from Theorem 2.1 of \cite{EC1}, the fourth one is just the definition of $\hat{\mu}^{n+1}$ and $\hat{T}_{n+1}$, and the fifth one is by the definition of $\eta^{n+1}_{\hat{T}_n}\eta^{n+1}_{\hat{T}_n}$.  By the universal property of the colimit and definition of $\kappa_n$, there is a unique natural transformation 
\[
\hat{\mu}:\hat{T}^2\Rightarrow\hat{T}
\]
such that 
\[
\begin{tikzpicture}[node distance=2cm,auto]
	\node (A) {$\hat{T}_n^2$};
	\node (B)[right of=A] {$\hat{T}^2$};
	\node(C)[below of=A] {$\hat{T}_n$};
	\node (D)[right of=C] {$\hat{T}$};
	\draw[->] (A) to node {$\kappa_n\kappa_n$}(B);
	\draw[->,swap] (A) to node{$\hat{\mu}^n$}(C);
	\draw[->] (B) to node{$\hat{\mu}$}(D);
	\draw[->,swap] (C) to node {$\kappa_n$}(D);
\end{tikzpicture}
\]
commutes for $n\geq 0$. In particular, we have that
\[
\begin{tikzpicture}[node distance=2cm,auto]
	\node (A) {$\hat{T}_0^2$};
	\node (B)[right of=A] {$\hat{T}^2$};
	\node(C)[below of=A] {$\hat{T}_0$};
	\node (D)[right of=C] {$\hat{T}$};
	\draw[->] (A) to node {$\kappa_0\kappa_0$}(B);
	\draw[->,swap] (A) to node{$\hat{\mu}^0$}(C);
	\draw[->] (B) to node{$\hat{\mu}$}(D);
	\draw[->,swap] (C) to node {$\kappa_0$}(D);
\end{tikzpicture}
\]
commutes, so that we must have that
\[
\begin{tikzpicture}[node distance=2cm,auto]
	\node (A) {$1_C$};
	\node (B)[right of=A] {$\hat{T}^2$};
	\node(C)[below of=A] {$1_C$};
	\node (D)[right of=C] {$\hat{T}$};
	\draw[->] (A) to node {$\hat{\eta}\hat{\eta}$}(B);
	\draw[->,swap] (A) to node{$1_{1_C}$}(C);
	\draw[->] (B) to node{$\hat{\mu}$}(D);
	\draw[->,swap] (C) to node {$\hat{\eta}$}(D);
\end{tikzpicture}
\]
commutes which says that
\[
\hat{\mu}\circ \hat{\eta}_{\hat{T}}\circ\hat{\eta}=\hat{\mu}\circ \hat{T}\hat{\eta}\circ\hat{\eta}
\]
\[
=\hat{\mu}\circ \hat{\eta}\hat{\eta}=\hat{\eta}.
\]
By Lemma \ref{ind_pointing_is_epi}, we are forced to have
\[
\hat{\mu}\circ \hat{\eta}_{\hat{T}}=1_{\hat{T}}.
\]
By naturality, we have that 
\[
\hat{\mu}\circ \hat{T}\hat{\eta}\circ\hat{\eta}=\hat{\mu}\circ \hat{\eta}_{\hat{T}}\circ\hat{\eta}.
\]
By Lemma \ref{ind_pointing_is_epi}, we are forced to have
\[
\hat{\mu}\circ \hat{T}\hat{\eta}=\hat{\mu}\circ \hat{\eta}_{\hat{T}}=1_{\hat{T}}.
\]
We now have that
\[
\hat{\mu}\circ \hat{\mu}_{\hat{T}}\circ\hat{\eta}_{\hat{T}^2}\circ\hat{\eta}_{\hat{T}}=\hat{\mu}\circ \hat{\eta}_{\hat{T}}
\]
\[
=1_{\hat{T}}=\hat{\mu}\circ \hat{\eta}_{\hat{T}}
\]
\[
=\hat{\mu}\circ \hat{T}\hat{\mu}\circ \hat{T}\hat{\eta}_{\hat{T}}\hat{\eta}_{\hat{T}}
\]
\[
=\hat{\mu}\circ \hat{T}\hat{\mu}\circ \hat{\eta}_{\hat{T}^2}\hat{\eta}_{\hat{T}}.
\]
Upon application of Lemma \ref{ind_pointing_is_epi}, we are forced to have that
\[
\hat{\mu}\circ \hat{\mu}_{\hat{T}}=\hat{\mu}\circ \hat{T}\hat{\mu},
\]
so that $T$ is a completable distributive series of monads.
\end{proof}

\noindent \small $\textbf{Acknowledgments}$ The author is deeply grateful to Simon Henry for his many illuminating comments about Grothendieck $\infty$-groupoids. Special thanks go out to Nick Gurski for recommending the use of distributive series of monads in my work, providing illuminating insights into higher dimensional structure data, and giving useful feedback to my work.
\vspace{0.3cm}

\noindent \small $\textbf{Funding}$ No funding was obtained for this research.
\vspace{0.3cm}

\noindent \large $\textbf{Declarations}$

\vspace{0.1cm}
\noindent \small $\textbf{Conflict of Interest}$ The author declares no conflict of interest.

\bibliographystyle{plain}
\bibliography{math}
\end{document}